\documentclass[12pt]{amsart}

\usepackage{amsmath,amssymb,latexsym,enumerate}
\usepackage{a4wide}
\usepackage[utf8]{inputenc}
\usepackage{color,graphicx}
\usepackage{ifthen}
\usepackage{verbatim}
\usepackage{fancyhdr}
\usepackage{url}
\usepackage{bbm}
\usepackage{graphicx}
\usepackage[sans]{dsfont}
\usepackage{mathrsfs}
\usepackage{hyperref}
\hypersetup{%
  colorlinks = true,
  linkcolor  = blue,
  citecolor    = red
}

\newcommand{\dist}{\operatorname{dist}}
\newcommand{\hess}{\operatorname{Hess}}

\newcommand{\ran}{\operatorname{Ran}}

\newcommand{\supp}{\operatorname{supp}}

\newcommand{\C}{\mathbb C}
\newcommand{\N}{\mathbb N}
\newcommand{\R}{\mathbb R}
\renewcommand{\S}{\mathbb S}

\newcommand{\diag}{\operatorname{diag}}
\newcommand{\bsigma}{\boldsymbol{\sigma}}
\def\<{\langle}
\def\>{\rangle}

\newcommand{\be}{\begin{equation}}
\newcommand{\ee}{\end{equation}}
\newcommand{\bes}{\begin{equation*}}
\newcommand{\ees}{\end{equation*}}

\numberwithin{equation}{section}
\numberwithin{figure}{section}

\newtheorem{theorem}{Theorem}[section]
\newtheorem{corollary}[theorem]{Corollary}
\newtheorem{lemma}[theorem]{Lemma}
\newtheorem{defin}[theorem]{Definition}
\newtheorem{prop}[theorem]{Proposition}

\newtheorem{example}[theorem]{Example}

\newtheorem{hyp}{Assumption}

\def\eee{{\mathcal E}}\def\fff{{\mathcal F}} \def\hhh{{\mathcal H}}
 \def\lll{{\mathcal L}}
 \def\ooo{{\mathcal O}}

\def\uuu{{\mathcal U}}

\begin{document}
\title{Metastable diffusions with degenerate drifts}

\author[M. Assal, J.-F. Bony, L. Michel]{Marouane Assal, Jean-Fran\c cois Bony and Laurent Michel}
\address{Marouane Assal, Departamento de Matem\'atica y Ciencia de la Computaci\'on, Universidad de Santiago de Chile, Las sophoras 173, Santiago, Chile.}
\email{marouane.assal@usach.cl}
\address{Jean-Fran\c{c}ois Bony, IMB, UMR CNRS 5251, Universit\'e de Bordeaux, 33405 Talence, France.}
\email{Jean-Francois.Bony@math.u-bordeaux.fr}
\address{Laurent Michel, IMB, UMR CNRS 5251, Universit\'e de Bordeaux, 33405 Talence, France.}
\email{laurent.michel@math.u-bordeaux.fr}

\begin{abstract}
We study the spectrum of the semiclassical Witten Laplacian $\Delta_{f}$ associated to a smooth function $f$ on $\R^d$. We assume that $f$ is a confining Morse--Bott function. Under this assumption we show that $\Delta_{f}$ admits exponentially small eigenvalues separated from the rest of the spectrum. Moreover, we establish Eyring-Kramers formula for these eigenvalues. Our approach is based on microlocal constructions of quasimodes near the critical submanifolds.
\end{abstract}

\maketitle

\tableofcontents 

\section{Introduction and main result} 

\subsection{Motivations}

The Witten Laplacian $\Delta_{f}$ associated to a smooth Morse function $f$ was introduced by Witten \cite{Wi82} to give an analytical proof of Morse inequalities. This operator appears also after unitary conjugation in the study of stochastic processes as the generator of overdamped Langevin dynamics associated to the drift $\nabla f$
\begin{equation} \label{e6}
d X_{t} = - 2 \nabla f ( X_{t} ) d t + \sqrt{2 h} d B_{t} ,
\end{equation}
where $X_{t}\in \R^{d}$ and $( B_{t} )_{t \geq 0}$ is a standard Brownian motion in $\R^{d}$. In this context, the semiclassical parameter is proportional to the temperature of the system and the study of the lowest eigenvalues of $\Delta_{f}$ gives crucial informations on the dynamic. In particular, the existence of exponentially small (with respect to $h^{- 1}$) eigenvalues of $\Delta_{f}$ explains the  metastable behavior in the low temperature regime. A detailed knowledge of the relevant time scales is also crucial in computational physics where ergodic Markov processes may be used to sample a target distributions and where many algorithms require a priori knowledge of the  metastable behavior \cite{Vo97,SoVo00}. We refer to \cite{LeRoSt10_01} for details on these topics.

The computation of the transition times of \eqref{e6} is a historical problem which at least goes back to Kramers \cite{Kr40}.
In the case of Morse functions, a  first rigorous study of the low eigenvalues of $\Delta_{f}$  was performed by Helffer and Sj\"ostrand \cite{HeSj85_01} who showed the correspondence between critical points of index $p$ of $f$ and exponentially small eigenvalues of the Witten Laplacian acting on $p$-forms. This approach was generalized to Morse--Bott inequalities in \cite{He88_02,HeSj88_01}. Later on, the first accurate computation of the exponential rate (Arrhenius law) and asymptotic expansion of the prefactor was done by Bovier, Gayrard and Klein \cite{BoGaKl05_01} by a probabilist approach and Helffer, Klein and Nier \cite{HeKlNi04_01} by semiclassical methods. More recently, Le Peutrec, Nier and Viterbo \cite{LePNiVi21} proved Arrhenius law for Lipschitz functions $f$ admitting a finite number of critical values.

In a more general framework, the study of the asymptotic behavior of the eigenvalues of Schr\"odinger operators of the form $P = - h^{2} \Delta + V ( x )$, in the semiclassical limit $h \to 0$, has a long history and has been the subject of several investigations from basis of quantum mechanics to microlocal analysis. Precise spectral asymptotics on the bottom of the spectrum has been proved for a large class of smooth real-valued potentials using the WKB method and harmonic approximations (we refer to \cite{DiSj99_01} for a detailed account). Under suitable assumptions, the low-lying eigenvalues are localized near the absolute minima of the potential $V$ and precise results on the splitting between eigenvalues can be obtained under additional geometric assumptions \cite{HeSj84_01,HeSj85_00}. At a first sight the analysis of the Witten Laplacian $\Delta_{f}$ associated to a Morse function $f$
requires even more sophisticated techniques, since it  presents non-resonant wells in the sense of \cite{HeSj85_02}. However it is possible to avoid the machinery  of  \cite{HeSj85_02} by using the existence of an explicit element in the kernel of $\Delta_{f}$ given by the Gibbs state $e^{- f / h}$. In \cite{HeKlNi04_01}, this is done by using additional supersymmetry properties and local analysis of the Witten Laplacian on $1$-forms. More recently,  a general construction of quasimodes based on Gaussian cut-off of the Gibbs state was developed in \cite{BoLePMi22} to study general Fokker--Planck operators.

In the present paper, we consider the case where the critical points of the function $f$ are made of smooth compact manifolds. 
This can be seen as an intermediate situation between the case of Morse function and the fully degenerate case of \cite{LePNiVi21}.
One of the motivations to work with submanifold critical sets comes also from physical context where symmetries in the problem yield such degenerate situations (see \cite{HeSj85_00} for operators invariant under a finite group of isometries). In particular, we provide the complete asymptotic of the small eigenvalues for radial functions $f$.

\subsection{Framework and first localization result}

Let $f : \R^{d} \rightarrow \R$ with $d \geq 1$ be a smooth function. We consider the associated semiclassical Witten Laplacian
\begin{equation} \label{e7}
\Delta_{f} = - h^{2} \Delta + \vert \nabla f \vert^{2} - h \Delta f ,
\end{equation}
where $h \in ] 0 , 1 ]$ denotes the semiclassical parameter. Throughout the paper, we assume that $f \in C^{\infty} ( \R^{d} ; \R )$ satisfies the following confining assumption.

\begin{hyp}\sl \label{hyp0}
There exist $C > 0$ and a compact set $K \subset \R^{d}$ such that
\begin{equation*}
f ( x ) \geq - C , \qquad \vert \nabla f ( x ) \vert \geq \frac{1}{C} \qquad \text{and} \qquad \vert \hess f ( x ) \vert \leq C \vert \nabla f ( x ) \vert^{2} ,
\end{equation*}
for all $x \in \R^{d} \setminus K$.
\end{hyp}

\noindent
Let us observe that, under Assumption \ref{hyp0}, there exist $C > 0$ and a compact set $L$ such that 
\begin{equation} \label{e8}
\forall x \in \R^{d} \setminus L , \qquad f ( x ) \geq C \vert x \vert ,
\end{equation}
see for example \cite[Lemma~3.14]{MeSc14} for a proof. Under this assumption, $\Delta_{f}$ is essentially self-adjoint on $C_{0}^{\infty} ( \R^{d} )$. By definition, $\Delta_{f}$ has a square structure
\begin{equation} \label{NN}
\Delta_{f} = d_{f}^{*} \circ d_{f} \qquad \text{with} \qquad d_{f} = e^{- f / h} \circ h \nabla \circ e^{f / h} ,
\end{equation}
which implies that $\Delta_{f}$ is non-negative and hence $\sigma ( \Delta_{f} ) \subset [ 0 , + \infty [$. Moreover, it follows from Assumption \ref{hyp0} that there exists $c_{0} , h_{0} > 0$ such that, for all $h \in ] 0 , h_{0} ]$,
\begin{equation} \label{ES}
\sigma_{\rm ess} ( \Delta_{f} ) \subset [ c_{0} , + \infty [ ,
\end{equation}
and hence $\sigma ( \Delta_{f} ) \cap [ 0 , c_{0} [$ is made of $h$-dependent discrete eigenvalues with no accumulation point excepted maybe $c_{0}$. In addition, \eqref{e8} gives that $e^{- f / h}$ belongs to the domain of $\Delta_{f}$ for all $h \in ] 0 , h_{0}]$, which implies thanks to \eqref{NN} that $0$ is a simple eigenvalue of $\Delta_{f}$.

The aim of this work is to describe the small eigenvalues of $\Delta_{f}$ in the degenerate case where $f$ is of Morse--Bott type. More precisely, throughout this paper we assume the following condition

\begin{hyp}\sl \label{hyp1}
The set of critical points of $f$ is a finite disjoint union of boundaryless compact connected submanifolds $\Gamma$ of $\R^{d}$ such that the transversal Hessian of $f$ at any point of $\Gamma$ is non degenerate. From now, we will denote by $\uuu$ the set of submanifolds $\Gamma$ as above and for any $\Gamma \in \uuu$ we denote $d_{\Gamma}$ its dimension.
\end{hyp}

Let us recall the celebrated Morse--Bott Lemma (see \cite{BaHu04} for a proof).

\begin{lemma}\sl \label{a36}
Assume that $f$ satisfies Assumption \ref{hyp1} and let $\Gamma \in \uuu$. Around any point of $\Gamma$, there exist local coordinates $( y_{t} , y_{-} , y_{+} )$ with $y_{t} \in \R^{d_{\Gamma}}$ and $( y_{-} , y_{+})\in \R^{d - d_{\Gamma}}$ such that
\begin{equation} \label{a3}
f ( y ) = - \vert y_{-} \vert^{2} + \vert y_{+} \vert^{2} .
\end{equation}
In particular, the signature of $\hess f$ is constant on $\Gamma$.
\end{lemma}

\noindent
Set
\begin{equation*}
\mathcal{U} = \bigcup_{j = 0}^{d} \uuu^{( j )} , \qquad \uuu^{( j )} : = \big\{ \Gamma \in \mathcal{U} ; \ \hess ( f )_{\vert \Gamma} \text{ has } j \text{ negative eigenvalues} \big\} ,
\end{equation*}
and, for $j = 0 , 1 , \ldots , d$, let $n_{j}$ be the cardinal of the set $\uuu^{( j )}$. Elements of $\uuu^{( 0 )}$ will be called minimal submanifolds and those of $\uuu^{( 1 )}$ will be called saddle submanifolds. Similarly to the Morse case, only the minimal manifolds can create small eigenvalues, and we have the following first localization result.

\begin{theorem}\sl \label{e14}
Let Assumptions \ref{hyp0} and \ref{hyp1} hold. There exist $\eta_{0} , h_{0} > 0$ such that, for all $h \in ] 0 , h_{0} ]$, $\Delta_{f}$ admits exactly $n_{0}$ eigenvalues in $[ 0 , \eta_{0} h^{2}]$ counting multiplicities, denoted $0 = \lambda_{1} ( h ) < \lambda_{2} ( h ) \leq \cdots \leq \lambda_{n_{0}} ( h )$. Furthermore, there exists a constant $c > 0$ such that, for all $j \in \{ 1 , \ldots , n_{0} \}$, one has 
\begin{equation*}
\lambda_{j} ( h ) = \mathcal{O} ( e^{- c / h} ) ,
\end{equation*}
uniformly for $h \in ] 0 , h_{0} ]$.
\end{theorem}

The proof of this result, based on the Helffer--Sj\"{o}strand theory of quantum wells \cite{HeSj84_01,HeSj86_01,HeSj87_01}, can be found in Section \ref{s2}. The aim of our paper is to give a precise description of the small eigenvalues $\lambda_{j} ( h )$, $j = 2 , \ldots , n_{0}$. More precisely, one aims to prove asymptotics of the form $\lambda_{j} ( h ) \sim a_{j} ( h ) e^{- 2 S_{j} / h}$ for some positive constants $S_{j}$ and some prefactors $a_{j}$ admitting an expansion in powers of $h$. Such asymptotics are often called Eyring--Kramers formula. In order to prove it, the first main difficulty is to identify the relevant energy barriers $S_{j}$. For this purpose, one needs to label the critical manifolds in a suitable way. This is the object of the next section.

\subsection{Separating saddle manifolds and labeling procedure} \label{s13}

For any $\sigma \in \R \cup \{ \infty \}$, let $X_{\sigma} = \{ x \in \R^{d} ; \ f ( x ) < \sigma \}$. Then $X_{\infty} = \R^{d}$ and, as soon as $n_{0} \geq 2$, there exists $\sigma \in \R$ such that $X_{\sigma}$ has at least two connected components. We now describe the structure of $X_{\sigma}$ near an element $\Gamma$ of $\uuu$ with $\sigma = f ( \Gamma )$. In the sequel, for $x_{0} \in \R^{d}$ and $r > 0$, $B ( x_{0} , r )$ stands for the open ball centered at $x_{0}$ and of radius $r$.

\begin{prop}\sl \label{a23}
Let $f$ satisfies Assumption \ref{hyp1} and denote $\sigma = f ( \Gamma )$ for $\Gamma\in \mathcal{U}$.

$i)$ For all $\Gamma \in \uuu^{( j )}$ with $j \geq 2$ and $r > 0$ small enough, the set $X_{\sigma} \cap ( \Gamma + B ( 0 , r ) )$ is connected.

$ii)$ For $\Gamma \in \uuu^{( 1 )}$, one of the following assertion holds    \setlength{\parskip}{0in}
\begin{itemize}
\item[$(a)$] either, for all $r > 0$ small enough, the set $X_{\sigma} \cap ( \Gamma + B ( 0 , r ) )$ is connected, \setlength{\parskip}{0.05in}

\item[$(b)$] or, for all $r > 0$ small enough, the set $X_{\sigma} \cap ( \Gamma + B ( 0 , r ) )$ has exactly two disjoint connected components $A_{+} ( r )$ and $A_{-} ( r )$. In that case, $\Gamma \subset \overline{A_{+} ( r )} \cap \overline{A_{-} ( r )} $.
\end{itemize}
\end{prop}

We postpone the proof of this proposition to Section \ref{s3}. Relying on this result, we introduce the following notions of locally separating and separating saddle manifolds.

\begin{defin}\sl \label{a2}
A saddle manifold $\Gamma \in \uuu^{( 1 )}$ satisfying $ii)$ $(b)$ of Proposition \ref{a23} is called {\it locally separating}. We say that a locally separating saddle manifold $\Gamma$ is {\it separating} when $A_{+} ( r )$ and $A_{-} ( r )$ belong to two disjoint connected components of $X_{\sigma}$ with $\sigma = f ( \Gamma )$. We will denote by $\uuu^{( 1 )}_{\rm sep}$ (resp. $\uuu^{( 1 )}_{\rm loc \ sep}$) the set of separating (resp. locally separating) saddle manifolds.
\end{defin}

Proposition \ref{a23} $i)$ shows that non-saddle critical manifolds are not separating. From Section 3.1 of \cite{HeKlNi04_01}, all the saddle points (that are saddle manifolds of dimension $0$) are locally separating (this also follows from \eqref{a27}). In dimension $1$ and $2$, all the saddle manifolds are locally separating. Indeed, we have just seen that this is the case when $d_{\Gamma} = 0$. Furthermore, if $d_{\Gamma} = 1$ in dimension $d = 2$, $\Gamma$ is topologically a circle which (globally) separates $\R^{2}$ into two parts. However, there exist saddle manifolds which are not locally separating in dimension $d \geq 3$, as shown by the following example.

\begin{figure}
\begin{center}
\includegraphics[width=230pt]{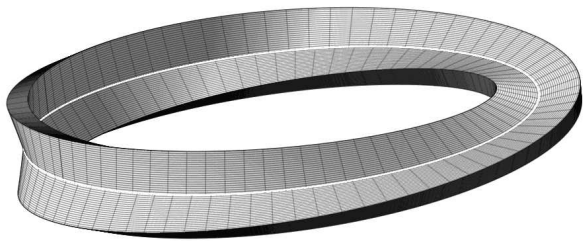}
\end{center}
\caption{The set $X_{\sigma}$ near the saddle manifold $\Gamma$ in Example \ref{e5}.} \label{f5}
\end{figure}

\begin{example}\rm \label{e5}
On $\R^{3}$ endowed with the cylinder variables $( r , \theta , z ) \in [ 0 , + \infty [ \times [ 0 , 2 \pi [ \times \R$, consider the function
\begin{equation*}
f = \big( ( r - 1 ) \cos ( \theta / 2) + z \sin ( \theta / 2 ) \big)^{2} - \big( z \cos ( \theta / 2 ) - ( r - 1 ) \sin ( \theta / 2 ) \big)^{2} ,
\end{equation*}
near $\Gamma = \{ ( r , \theta , z ) ; \ r = 1 \text{ and } z = 0\}$. This is noting than the function $a^{2} - b^{2}$ apply to the vector $( r - 1 , z )$ after a rotation of angle $\theta / 2$. Thus, $f$ is smooth, satisfies Assumption \ref{hyp1} and $\Gamma \in \uuu^{( 1 )}$. But, since the rotation of angle $\theta / 2$ induces a symmetry after a turn along $\Gamma$, this saddle manifold is not locally separating (see Figure \ref{e5}).
\end{example}

We deduce from Proposition \ref{a23} the following statement.

\begin{lemma}\sl \label{a4}
Let $\Gamma \in \uuu^{( 1 )}_{\rm sep}$ and $\sigma = f ( \Gamma )$. There exist exactly two connected components $B_{\pm}$ of $X_{\sigma}$ such that $\Gamma \cap \overline{B_{\pm}} \neq \emptyset$. Moreover, $\Gamma \subset \overline{B_{+}} \cap \overline{B_{-}}$ (see Figure \ref{f1}).
\end{lemma}

With Definition \ref{a2} in mind, we can adapt the labeling procedure of minima and saddle manifolds introduced in \cite{HeKlNi04_01} and generalized in Section 4 of \cite{HeHiSj11_01}. There is no difference here expect that the role of the saddle points in the Morse case is replaced by the locally separating saddle manifolds in the present setting. Following the presentation of \cite{Mi19}, we recall quickly this labeling procedure and send the reader to the previous references for more details.

The set $\Sigma_{\rm sep}$ of separating saddle values is defined by $\Sigma_{\rm sep} = \{ f ( \Gamma ) ; \ \Gamma \in \uuu^{( 1 )}_{\rm sep} \}$. Its elements arranged in the decreasing order are denoted $\sigma_{2} > \sigma_{3} > \cdots > \sigma_{N}$ to which is added a fictive infinite separating saddle value $\sigma_{1} = + \infty$. Starting from $\sigma_{1}$, we will successively associate to each $\sigma_{i}$ a finite family of local minimal manifolds $( m_{i , j} )_{j}$ and a finite family of connected components $( E_{i , j} )_{j}$ of $X_{\sigma_{i}}$.

We choose $m_{1 , 1}$ as any global minimal manifold of $f$ (not necessarily unique) and $E_{1 , 1} = \R^{d}$. In the sequel, we denote ${\underline m} = m_{1 , 1}$. We continue the labeling procedure by induction and suppose that the families $( m_{k , j} )_{j}$ and $( E_{k , j} )_{j}$ have been constructed for all $1 \leq k \leq i - 1$. The set $X_{\sigma_{i}} = \{ x \in \R^{d} ; \ f ( x ) < \sigma_{i} \}$ has finitely many connected components and we label $E_{i , j}$, $j = 1 , \ldots , N_{i}$ those of these components that do not contain any $m_{k , \ell}$ with $k < i$. In each $E_{i , j}$ we pick up a minimal manifold $m_{i , j}$ which is a global minimum of $f_{\vert E_{i , j}}$. We run the procedure until all the minimal manifolds have been labeled. Note that all the components $E_{i , j}$ with $i \geq 2$ are critical in the sense that there exists $\Gamma \in \uuu^{( 1 )}_{\rm sep}$ such that $\Gamma \subset \overline{E_{i , j}}$ (see Lemma \ref{a4}).

\begin{figure}
\begin{center}
\begin{picture}(0,0)%
\includegraphics{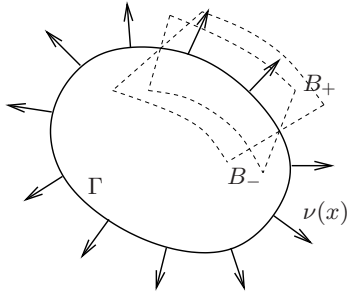}%
\end{picture}%
\setlength{\unitlength}{1184sp}%
\begingroup\makeatletter\ifx\SetFigFont\undefined%
\gdef\SetFigFont#1#2#3#4#5{%
  \reset@font\fontsize{#1}{#2pt}%
  \fontfamily{#3}\fontseries{#4}\fontshape{#5}%
  \selectfont}%
\fi\endgroup%
\begin{picture}(6966,6141)(-3407,-7144)
\put(-1649,-5086){\makebox(0,0)[lb]{\smash{{\SetFigFont{9}{10.8}{\rmdefault}{\mddefault}{\updefault}$\Gamma$}}}}
\put(2851,-2911){\makebox(0,0)[lb]{\smash{{\SetFigFont{9}{10.8}{\rmdefault}{\mddefault}{\updefault}$B_{+}$}}}}
\put(1276,-4861){\makebox(0,0)[lb]{\smash{{\SetFigFont{9}{10.8}{\rmdefault}{\mddefault}{\updefault}$B_{-}$}}}}
\put(2851,-5611){\makebox(0,0)[lb]{\smash{{\SetFigFont{9}{10.8}{\rmdefault}{\mddefault}{\updefault}$\nu ( x )$}}}}
\end{picture}%
\end{center}
\caption{The geometry near a separating saddle manifold.}
\label{f1}
\end{figure}

Throughout $\Gamma_{1}$ is a fictive saddle point such that $f ( \Gamma_{1} ) = \sigma_{1} = + \infty$  and for any set $A$, ${\mathcal P} ( A )$ denotes the power set of $A$. From the above labeling, we define two mappings
\begin{equation}\label{FJ}
E : {\mathcal U}^{( 0 )} \to {\mathcal P} ( \R^{d} ) \qquad \text{and} \qquad {\bf j} : {\mathcal U}^{( 0 )} \to {\mathcal P} ( \uuu^{( 1 )}_{\rm sep} \cup \{ \Gamma_{1} \} ) ,
\end{equation}
as follows: for every $i \in \{ 1 , \dots , N \}$ and $j \in \{ 1 , \dots , N_{i} \}$,
\begin{equation} \label{a54}
E ( m_{i , j} ) : = E_{i , j} ,
\end{equation}
and
\begin{equation} \label{a53}
{\bf j} ( {\underline m} ) : = \{ \Gamma_{1} \} \qquad \text{and} \qquad {\bf j} ( m_{i , j} ) : = \{\Gamma\in \uuu_{\rm sep}^{( 1 )} ; \ \Gamma \cap \partial E_{i , j} \neq \emptyset\} \text{ for } i \geq 2 .
\end{equation}
In particular, we have $E ( \underline{m} ) = \R^{d}$ and, for all $i, j \in \{ 1 , \dots , N \}$, one has ${\bf j} ( m_{i , j} ) \neq \emptyset$ and $f_{\vert \Gamma} = \sigma_{i}$ for all $\Gamma \in {\bf j} ( m_{i , j} )$. Moreover, it follows from Lemma \ref{a4} that 
\begin{equation} \label{e9}
\forall m \in \uuu^{( 0 )} \setminus \{ {\underline m} \} , \qquad {\bf j} ( m ) \subset {\mathcal P} ( \partial E ( m ) ) .
\end{equation}
We then define the mappings
\begin{equation}\label{e11}
{\bsigma} : \uuu^{( 0 )} \rightarrow f ( \uuu_{\rm sep}^{( 1 )} ) \cup \{ \sigma_{1} \} \qquad \text{and} \qquad S : \uuu^{( 0 )} \rightarrow ] 0 , + \infty ] ,
\end{equation}
by
\begin{equation} \label{a55}
\forall m \in \uuu^{( 0 )} , \qquad \bsigma ( m ) : = f ( {\bf j} ( m ) ) \qquad \text{and} \qquad S ( m ) : = \bsigma ( m ) - f ( m ) ,
\end{equation}
where, with a slight abuse of notation, we have identified the set $f ( {\bf j} ( m ) )$ with its unique element. Note that $S ( m ) = + \infty$ if and only if $m = \underline{m}$.

\subsection{Main result} We are now in position to introduce our last assumption. In addition to Assumptions \ref{hyp0} and \ref{hyp1}, we will suppose

\begin{hyp}\sl \label{hyp2} The following holds true:
\begin{align*}
&\star \text{ for all } m \in \uuu^{( 0 )} , \ m \text{ is the unique global minimum of } f_{\vert E ( m )} ,  \\
&\star \text{ for all } m , m^{\prime} \in \uuu^{( 0 )} \text{ with } m \neq m ^{\prime} , \ {\bf j} ( m ) \cap {\bf j} ( m^{\prime} ) = \emptyset.
\end{align*}
\end{hyp}

In particular, this assumption implies that $f$ uniquely attains its global minimum at ${\underline m} \in \uuu^{( 0 )}$. This assumption is a generalization of Hypothesis 5.1 of \cite{HeHiSj11_01} (see also \cite{LePMi20}). By smooth perturbation of the function $f$, one sees that it is generically satisfied. We now state our main result.

\begin{theorem}\sl \label{e12}
Let Assumptions \ref{hyp0}, \ref{hyp1}, \ref{hyp2} hold. There exist $\eta_{0} , h_{0} > 0$ such that, for all $h \in ] 0 , h_{0}]$, one has, counting the eigenvalues with multiplicities,
\begin{equation*}
\sigma ( \Delta_{f} ) \cap [ 0 , \eta_{0} h^{2} ] = \{ \lambda ( m , h) ; \ m \in \mathcal{U}^{( 0 )} \} ,
\end{equation*}
where $\lambda ( \underline{m} , h ) = 0$ and, for all $\underline{m} \neq m \in \mathcal{U}^{( 0 )}$, $\lambda ( m , h )$ satisfies the following Eyring--Kramers law
\begin{equation}\label{e13}
\lambda ( m , h ) = \mathcal{D} ( m ) h^{\frac{d_{m} - d_{m}^{\rm max}}{2} + 1} e^{- \frac{2 S ( m )}{h}} \alpha ( h ) ,
\end{equation}
where $S: \mathcal{U}^{(0)} \to ] 0 , + \infty ]$ is defined by \eqref{e11}, $\alpha ( h )$ admits a classical expansion in powers of $h^\frac{1}{2}$ of the form $\alpha ( h ) \sim \sum_{j \geq 0} \alpha_{j} h^{\frac{j}{2}}$ with $\alpha_{0} = 1$ and $(\alpha_{j})_{j \geq 1} \subset \R$, and for any $m \in \mathcal{U}^{( 0 )} \setminus \{ \underline{m} \}$
\begin{equation*}
\mathcal{D} ( m ) = \frac{\sum_{\Gamma \in {\bf{j}}^{\text{max}} ( m )} \pi^{\frac{d_{m} - d_{m}^{\rm max}}{2} - 1} \int_{\Gamma} \vert \mu ( s ) \vert \vert \det \hess_{\perp} f ( s ) \vert^{- \frac{1}{2}} d s}{\int_{m} \vert \det \hess_{\perp} f ( s ) \vert^{- \frac{1}{2}} d s} .
\end{equation*}
Here $d_{m}^{\rm max} : = \max_{\Gamma \in {\bf j} ( m )} d_{\Gamma}$, ${\bf j}^{\text{max}} ( m ) : = \{ \Gamma \in {\bf{j}} ( m ) ; \ d_{\Gamma} = d_{m}^{\rm max} \}$, $\mu ( s )$ is the unique negative eigenvalue of $\hess f ( s )$ and $\hess_{\perp} f$ is the Hessian of $f$ restricted to the normal space of the considered critical manifold.
\end{theorem}

Recall that the equation
\begin{equation}\label{eq:evol1}
h \partial_{t} u + \Delta_{f} u = 0 , \qquad u_{\vert t = 0} = u_{0} ,
\end{equation}
models the evolution of the probability of presence of a Brownian particle, solution of \eqref{e6}, with initial distribution $u_0$. Hence
the spectral asymptotics of the above theorem yield immediately quantitative informations on the solutions of \eqref{eq:evol1}.
First,  the time to return to equilibrium is given by the inverse of the spectral gap, that is the inverse of the first non zero eigenvalue. Moreover, the precise knowledge of the other small eigenvalues 
permits to understand the metastable behavior of the system (see Corollary 1.6 in \cite{BoLePMi22} for precise statements in the context of Morse functions).

The Eyring--Kramers asymptotic \eqref{e13} has some similarities with the one obtained in the Morse case, see \cite{BoGaKl05_01, HeKlNi04_01}. We recognize the same exponentially small factor $e^{- 2 S ( m ) / h}$, however, the power of $h$ depends now on the dimension of the minimal manifolds $m$ and of the separating saddle manifolds $\Gamma$. Lastly, the constant factor $\mathcal{D} ( m )$ averages the contribution of the critical manifolds $m$ and $\Gamma$. In the more general setting where the function $f$ is only assumed to be Lipschitz subanalytic with a finite number of critical values, Le Peutrec, Nier and Viterbo \cite{LePNiVi21} are able to give the semiclassical limit of $h \ln \lambda ( m , h)$. Though, this approach is very general, it seems that, in many geometrical cases, it doesn't allow to recover the prefactor and in particular, the power of $h$ in the asymptotic of $\lambda ( m , h )$.

As already noticed, the  return to the global equilibrium is  faster as the spectral gap is larger. We observe from the power of $h$ in the prefactor of \eqref{e13} that the spectral gap increases when $d_{m}$ decreases or $d_{m}^{\rm max}$ increases. This is natural since a minimal manifold of smaller dimension seems less trapping. The same way, it seems easier to pass through a saddle manifold of larger dimension. In this direction, note also that only the saddle manifolds of maximal dimension appear in the leading term of \eqref{e13} which suggests that the underlying process selects the largest saddle manifolds to escape from a local minimum. Variations of the power of $h$ were observed in \cite{BeGe10} where the authors prove Eyring--Kramers formula for the exit time from a domain in the case of non quadratic separating saddle point. It would be very interesting to give rigorous results on the exit event in our Morse--Bott case in the spirit of \cite{BoEcGaKl04,DiGLeLePNe19}.

In the usual Eyring--Kramers asymptotic for Morse functions, the corresponding symbol $\alpha(h)$ admits generally an asymptotic expansion in integer powers of $h$. In our Morse--Bott setting, this is the case if and only if the dimensions of the saddle manifolds in ${\bf j}(m)$ have the same parity. This follows from Proposition \ref{a42} and \eqref{e22}.

The results have been stated on $\R^{d}$, but can be adapted to compact boundaryless manifolds. Indeed, the constructions made near the saddle manifolds are purely local. On the other hand, let us observe that the generic Assumption \ref{hyp2} could certainly be relaxed by using methods in the spirit of \cite{Mi19,BoLePMi22}. Finally, generalizations to the study of Witten Laplacian on $p$-forms, $p\geq 1$, could also be investigated.

\subsection{Applications and examples}

In this part, we study typical situations where Theorem \ref{e12} can be used.

We first give the asymptotic of the small eigenvalues of $\Delta_{f}$ when $f$ is radial. Then, let $f ( x )$ be a radial smooth function on $\R^{d}$, $d \geq 2$, satisfying Assumptions \ref{hyp0} and \ref{hyp1}. Outside of $0$, the critical manifolds of $f$ are spheres of dimension $d - 1$ which are either minima or separating saddles. Moreover, $0$ is either a minimum or a maximum of dimension $0$. We define
\begin{equation*}
F ( r ) = f ( x ) \qquad \text{with} \qquad r = \sqrt{x^{2}_{1} + \cdots + x^{2}_{d}} ,
\end{equation*}
for $r \in ( 0 , + \infty [$ with $( = [$ or $( = ]$ if $0$ is a minimum or a maximum respectively. The function $F$ is a smooth Morse function satisfying Assumption \ref{hyp0} on $( 0 , + \infty [$. Furthermore, the minima and saddles of $F$ correspond to those of $f$. The labeling procedures at the end of Section \ref{s13} for $f$ and $F$ can be carried out in parallel and can lead to the same labels mutatis mutandis: a critical point $r \in ( 0 , + \infty [$ of $F$ corresponds to a critical manifold $\{ \vert x \vert = r \}$ of $f$. Thus, $f$ satisfies Assumption \ref{hyp2} if and only if $F$ satisfies Assumption \ref{hyp2} on $( 0 , + \infty [$. Let us suppose that this is the case in the sequel.

We know that ${\bf j} ( m ) \neq \emptyset$ for all minimum $m$ (see below \eqref{a26}). Moreover, the number of minima and saddle points is the same in dimension $1$ (recall that a fictive saddle point has been added). Thus, the second part of Assumption \ref{hyp2} implies that, for any minimal manifold $m = \{ \vert x \vert = r_{m} \}$ with $r_{m} \in ( 0 , + \infty [$, there exists a unique saddle manifold $\Gamma_{m} = \{ \vert x \vert = s_{m} \}$ with $s_{m} \in ] 0 , + \infty [$ such that ${\bf j} ( m ) = \{ \Gamma_{m} \}$. Theorem \ref{e12} directly gives

\begin{corollary}[Asymptotic for radial functions]\sl \label{e4}
In the previous setting,
\begin{equation} \label{e3}
\lambda ( m , h ) =
\left\{ \begin{aligned}
&\frac{s_{m}^{d - 1}}{\pi r_{m}^{d - 1}} \sqrt{F^{\prime \prime} ( r_{m} ) \vert F^{\prime \prime} ( s_{m} ) \vert} h e^{- \frac{2 S ( m )}{h}} \alpha ( h ) &&\text{for } m \neq \{ 0 \} , \\
&\frac{\vert \S^{d - 1} \vert s_{m}^{d - 1}}{\pi^{\frac{1 + d}{2}}} F^{\prime \prime} ( r_{m} )^{\frac{d}{2}} \sqrt{\vert F^{\prime \prime} ( s_{m} ) \vert} h^{\frac{3 - d}{2}} e^{- \frac{2 S ( m )}{h}} \alpha ( h )&&\text{for } m = \{ 0 \} ,
\end{aligned} \right.
\end{equation}
where $\vert \S^{d - 1} \vert$ denotes the measure of the unit sphere, $S ( m ) = F ( s_{m} ) - F ( r_{m} )$ and $\alpha ( h )$ is as in \eqref{e13}. The second part of \eqref{e3} makes sense only when $0$ is a minimum.
\end{corollary}

This asymptotic can be compared with the one associated to $F$ in dimension $1$. More precisely, let $\Lambda ( m , h )$ denote the quantity formally computed from formula \eqref{e13} with the function $F$ at the minimal point $r_{m}$. For $m = \{ 0 \}$, which only makes sense when $0$ is a minimum of $f$, this computation is purely formal. For $m \neq \{ 0 \}$, this quantity can be seen as the eigenvalue of $\Delta_{\widetilde{F}}$ where the function $\widetilde{F}$ is defined on $\R$, satisfies the assumptions of Theorem \ref{e12} and coincides with $F$ outside a small neighborhood of $]- \infty , 0 ]$ without additional critical point. Then, we have
\begin{equation}
\Lambda ( m , h ) \sim \lambda ( m , h ) \times
\left\{ \begin{aligned}
&\frac{r_{m}^{d - 1}}{s_{m}^{d - 1}} &&\text{for } m \neq \{ 0 \} , \\
&\pi^{\frac{d - 1}{2}} F^{\prime \prime} ( r_{m} )^{\frac{1 - d}{2}} s_{m}^{1 - d}  \vert \S^{d - 1} \vert^{- 1} h^{\frac{d - 1}{2}} &&\text{for } m = \{ 0 \} .
\end{aligned} \right.
\end{equation}
Roughly speaking, spherical minima behave like minimal points en dimension one whereas $0$ yields an asymptotic of different order. We now apply Corollary \ref{e4} in a concrete situation.

\begin{figure}
\begin{center}
\includegraphics[width=180pt]{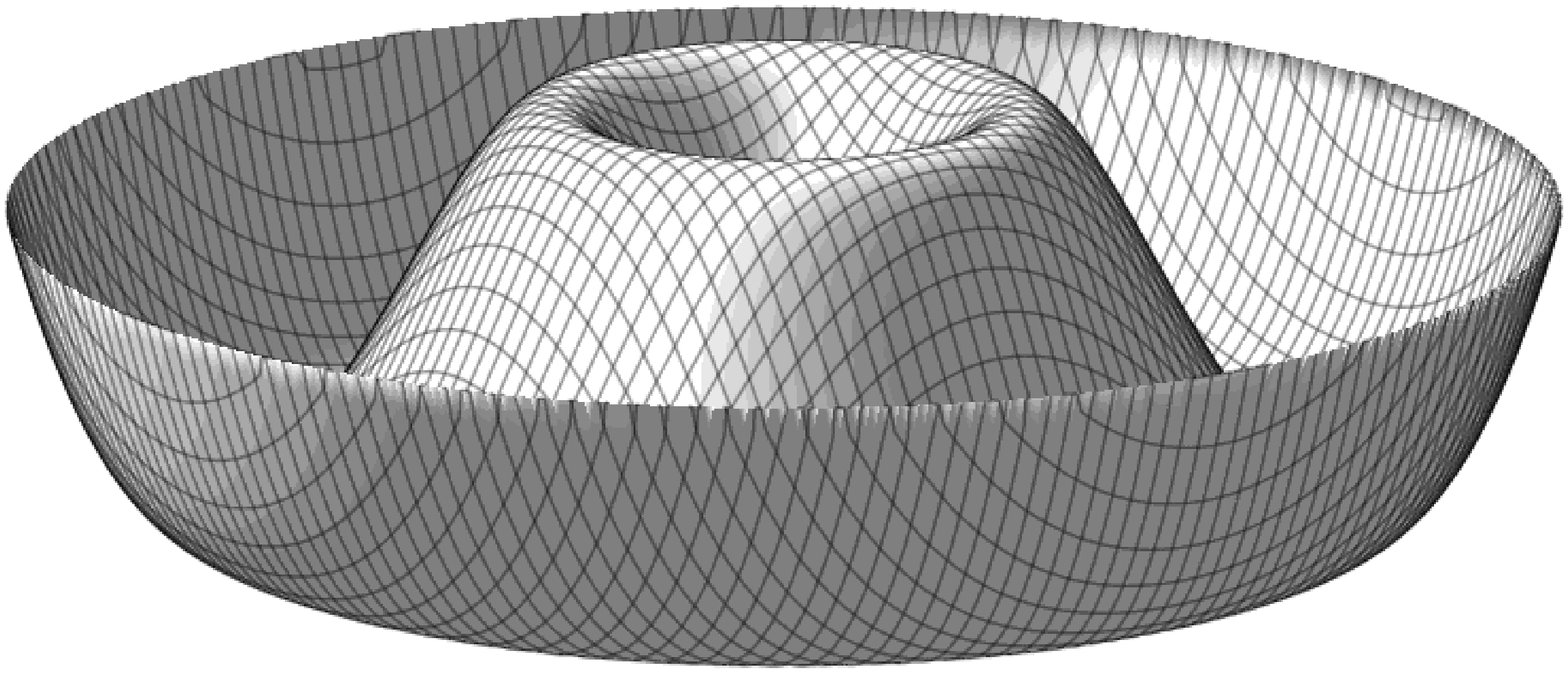} $\quad \qquad$
\begin{picture}(0,0)%
\includegraphics{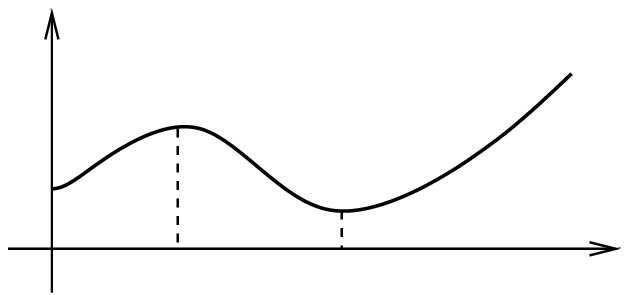}%
\end{picture}%
\setlength{\unitlength}{1381sp}%
\begingroup\makeatletter\ifx\SetFigFont\undefined%
\gdef\SetFigFont#1#2#3#4#5{%
  \reset@font\fontsize{#1}{#2pt}%
  \fontfamily{#3}\fontseries{#4}\fontshape{#5}%
  \selectfont}%
\fi\endgroup%
\begin{picture}(8598,4005)(-764,206)
\put(7801,1064){\makebox(0,0)[b]{\smash{{\SetFigFont{9}{10.8}{\rmdefault}{\mddefault}{\updefault}$r$}}}}
\put(3976,389){\makebox(0,0)[b]{\smash{{\SetFigFont{9}{10.8}{\rmdefault}{\mddefault}{\updefault}$r_{1}$}}}}
\put(4051,1664){\makebox(0,0)[b]{\smash{{\SetFigFont{9}{10.8}{\rmdefault}{\mddefault}{\updefault}$m_{1}$}}}}
\put(1726,2789){\makebox(0,0)[b]{\smash{{\SetFigFont{9}{10.8}{\rmdefault}{\mddefault}{\updefault}$\Gamma_{2}$}}}}
\put(220,4064){\makebox(0,0)[lb]{\smash{{\SetFigFont{9}{10.8}{\rmdefault}{\mddefault}{\updefault}$F ( r )$}}}}
\put(1726,389){\makebox(0,0)[b]{\smash{{\SetFigFont{9}{10.8}{\rmdefault}{\mddefault}{\updefault}$s_{2}$}}}}
\put(-299,389){\makebox(0,0)[b]{\smash{{\SetFigFont{9}{10.8}{\rmdefault}{\mddefault}{\updefault}$0$}}}}
\put(-749,1664){\makebox(0,0)[lb]{\smash{{\SetFigFont{9}{10.8}{\rmdefault}{\mddefault}{\updefault}$m_{2}$}}}}
\end{picture}%
\end{center}
\caption{The functions $f ( x )$ and $F ( r )$ in Example \ref{e1}.} \label{f4}
\end{figure}

\begin{example}[Mexican hat]\rm \label{e1}
We consider a smooth function $f$ on $\R^{2}$ which is radial, satisfies the assumptions of Theorem \ref{e12} and is as in Figure \ref{f4}. Other types of Mexican hats have been considered around Figure 19 of \cite{LePNiVi21}. In the present setting, the set of critical manifolds writes
\begin{equation*}
\uuu = \{ m_{1} , m_{2} , \Gamma_{2} \} ,
\end{equation*}
where, using the notations at the end of Section \ref{s13},
\begin{equation*}
m_{1} = m_{1 , 1} = \underline{m} = \{ r = r_{1} \} , \qquad m_{2} = m_{2 , 1} = \{ 0 \} , \qquad \Gamma_{2} = \{ r = s_{2} \} ,
\end{equation*}
and ${\bf j} ( m_{2} ) = \{ \Gamma_{2} \}$. Note that $d_{m_{1}} = d_{\Gamma_{2}} = 1$ and $d_{m_{2}} = 0$. As before, we define $F ( r ) = f ( x )$ for $r \in [ 0 , + \infty [$. From Corollary \ref{e4}, the two exponentially small eigenvalues of $\Delta_{f}$ satisfy $\lambda ( m_{1} , h ) = 0$ and
\begin{equation}
\lambda ( m_{2} , h ) = \frac{2 s_{2}}{\sqrt{\pi}} F^{\prime \prime} ( 0 ) \sqrt{\vert F^{\prime \prime} ( s_{2} ) \vert} \sqrt{h} e^{- \frac{2 S ( m_{2} )}{h}} \alpha ( h ) ,
\end{equation}
with $S ( m_{2} ) = F ( s_{2} ) - F ( 0 )$ and $\alpha ( h )$ as in \eqref{e13}.
\end{example}

We finish this section with another type of examples.

\begin{example}[Blow-up of minima]\rm \label{e2}
Let $\widetilde{f}$ be a Morse function on $\R^{d}$, $d \geq 2$, satisfying Assumptions \ref{hyp0} and \ref{hyp2}. It is possible to construct a (not unique) smooth function $f$ ``blowing up'' the minima of $\widetilde{f}$. It means that $f = \widetilde{f}$ outside of a neighborhood of the minima of $\widetilde{f}$ and that each minimal point $\widetilde{m}$ of $\widetilde{f}$ becomes a small minimal manifold $m$ of $f$ diffeomorphic to the sphere $\S^{d - 1}$ with $\widetilde{f} ( \widetilde{m} ) = f ( m )$ (see Figure \ref{f6}).

\begin{figure}
\begin{center}
\begin{picture}(0,0)%
\includegraphics{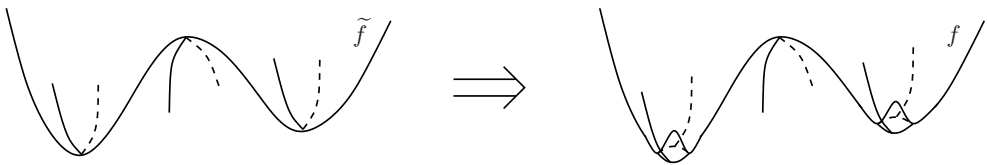}%
\end{picture}%
\setlength{\unitlength}{1381sp}%
\begingroup\makeatletter\ifx\SetFigFont\undefined%
\gdef\SetFigFont#1#2#3#4#5{%
  \reset@font\fontsize{#1}{#2pt}%
  \fontfamily{#3}\fontseries{#4}\fontshape{#5}%
  \selectfont}%
\fi\endgroup%
\begin{picture}(17616,2991)(-5057,581)
\put(1201,2939){\makebox(0,0)[lb]{\smash{{\SetFigFont{9}{10.8}{\rmdefault}{\mddefault}{\updefault}$\widetilde{f}$}}}}
\put(11851,2939){\makebox(0,0)[lb]{\smash{{\SetFigFont{9}{10.8}{\rmdefault}{\mddefault}{\updefault}$f$}}}}
\end{picture}%
\end{center}
\caption{A blow-up of minima described in Example \ref{e2}.}
\label{f6}
\end{figure}

Then, the critical manifolds of $f$ are those of $\widetilde{f}$ except that the minimal points $\widetilde{m}$ of $\widetilde{f}$ are replaced by these small manifolds $m$ and that there is an additional local maximum inside each of these manifolds. In particular, $f$ is a Morse--Bott function and satisfies Assumptions \ref{hyp0} and \ref{hyp1}. Moreover, the labeling procedure at the end of Section \ref{s13} is the same for $\widetilde{f}$ and $f$ (except that $\widetilde{m}$ is replaced by $m$), showing that Assumption \ref{hyp2} holds.

Let $\lambda ( m , h )$ (resp. $\widetilde{\lambda} ( \widetilde{m} , h)$) denote the exponentially small eigenvalues of $\Delta_{f}$ (resp. $\Delta_{\widetilde{f}})$. Theorem \ref{e12} provides the relation
\begin{equation}
\lambda ( m , h ) \sim \alpha_{m} h^{\frac{d - 1}{2}} \widetilde{\lambda} ( \widetilde{m} , h ) ,
\end{equation}
for some constant $\alpha_{m} \in ] 0 , + \infty [$. This discussion is still valid if we only assume that the minima (and not all the critical manifolds) of $\widetilde{f}$ are points.
\end{example}

The plan of the paper is the following. In the next section, we give a proof of Theorem \ref{e14}. Section \ref{s3} is devoted to microlocal constructions near the saddle submanifolds. These constructions are used in Section \ref{s4} to define locally the  quasimodes. In Section \ref{s5} we  glue these quasimodes with parts of the global Gibbs state to construct global quasimodes.  In the last section, we build the interaction matrix and prove Theorem \ref{e12}.

\section{Proof of Theorem \ref{e14}}\label{s2}

This result is a consequence of the works of Helffer and Sj\"ostrand \cite{HeSj84_01,HeSj86_01,HeSj87_01}. Following these papers, we introduce the Agmon metric $\vert \nabla f ( x ) \vert^{2} d x^{2}$, where $d x^{2}$ denotes the Euclidean metric on $\R^{d}$, and let $d_{\rm Ag} ( x , y )$ be the associated degenerate distance on $\R^d$. Given $\Gamma \in \mathcal{U}$, the Agmon distance to $\Gamma$ is defined by
\begin{equation}\label{e15}
\varphi_{\Gamma} ( x ) = d_{\rm Ag} ( x , \Gamma ) : = \inf_{y \in \Gamma} d_{\rm Ag} ( x , y ) .
\end{equation}
Recall that $\varphi_{\Gamma}$ is a non-negative smooth function in a neighborhood of $\Gamma$ which vanishes exactly at the order $2$ on $\Gamma$ and satisfies $\vert \nabla \varphi_{\Gamma} \vert^{2} = \vert \nabla f \vert^{2}$ (see for instance Section $0$ of \cite{HeSj86_01}). Let $( M_{\Gamma} )_{\Gamma \in \mathcal{U}}$ be a family of small compact neighborhoods of $\Gamma \in \mathcal{U}$ and consider the self-adjoint realization $P_{M_{\Gamma}}$ of $\Delta_{f}$ on $M_{\Gamma}$ with Dirichlet boundary conditions.

Let $\Gamma\in \uuu$ be a critical submanifold. Suppose first that $\Gamma \in \uuu^{( 0 )}$. In that case, $\varphi_{\Gamma} = f - f ( \Gamma )$ in a neighborhood of $\Gamma$ and then $\Delta \varphi_{\Gamma} - \Delta f = 0$ near $\Gamma$. Thus, applying Theorem 2.3 of \cite{HeSj87_01} with $E_{0} = E_{1} = 0$ and $E_{2} : = \inf ( \sigma ( P_{\Gamma} ) ) = 0$, there exist $\eta_{\Gamma} , h_{\Gamma} > 0$ such that for all $h \in ] 0 , h_{\Gamma} ]$, the spectrum of $P_{M_{\Gamma}}$ in $] - \infty , \eta_{\Gamma} h^{2} ]$ is reduced to a simple eigenvalue $\lambda_{\Gamma} ( h )$. Here, the operator $P_{\Gamma} = \nabla_{\Gamma}^{( \widetilde{*} )} \nabla_{\Gamma}$ is defined in equation (1.5) of \cite{HeSj87_01} using $\Delta \varphi_{\Gamma} - \Delta f = 0$. Moreover, for $\chi \in C_{0}^{\infty} ( M_{\Gamma} )$ with $\chi = 1$ near $\Gamma$, one has 
\begin{equation*}
P_{M_ {\Gamma}} ( \chi e^{- ( f - f ( \Gamma ) ) / h} ) = \Delta_{f} ( \chi e^{- ( f - f ( \Gamma ) / h} ) = \ooo ( e^{- c / h} ) \qquad \text{and} \qquad \Vert \chi e^{- ( f - f ( \Gamma ) ) / h} \Vert \gtrsim h^{\frac{d - d_{\Gamma}}{4}} ,
\end{equation*}
for some constant $c > 0$ and $h > 0$ small enough. Consequently, $\lambda_{\Gamma} ( h )$ is exponentially small with respect to $h$, that is $\lambda_{\Gamma} ( h ) = \ooo ( e^{- c / h} )$ for some $c > 0$.

Suppose now that $\Gamma \in \uuu \setminus \uuu^{( 0 )}$ and set $T : = \min_{x \in \Gamma} ( \Delta \varphi_{\Gamma} - \Delta f ) ( x )$. Using $\vert \nabla \varphi_\Gamma \vert^{2} = \vert \nabla f \vert^{2}$ near $\Gamma$ and $\hess \varphi_\Gamma \geq 0$ on $\Gamma$, we deduce $\hess \varphi_{\Gamma} ( x ) = \vert \hess f ( x ) \vert$ for all $x \in \Gamma$. Since $\Delta g = {\rm tr} \hess g$, it yields 
\begin{equation*}
T = 2 \min_{x \in \Gamma} \sum_{\substack{\mu ( x ) \in \sigma ( \hess f ( x ) ) \\ \mu ( x ) < 0}} \vert \mu ( x ) \vert > 0 .
\end{equation*}
The Melin--H\"ormander inequality, more precisely Proposition 2.1 of \cite{HeSj86_01}, applied to the operator $Q := \Delta_{f} - h T$ gives 
\begin{equation*}
\< Q u , u \> \geq h^{3} \Vert \nabla u \Vert^{2} - C h^{2} \Vert u \Vert^{2} \geq - C h^{2} \Vert u \Vert^{2} .
\end{equation*}
Thus, there exists $h_{\Gamma} > 0$ such that
\begin{equation*}
\forall u \in C_{0}^{\infty} ( M_{\Gamma} ) , \qquad \< P_{M_{\Gamma}} u , u \> \geq \frac{h T}{2} \Vert u \Vert^{2} \geq h^{2} \Vert u \Vert^{2} ,
\end{equation*}
and hence $\sigma ( P_{M_{\Gamma}} ) \subset [ h^{2} , + \infty [$ for all $h \in ] 0 , h_{\Gamma} ]$.

Define the constants $h_{0} : = \min_{\uuu} h_{\Gamma} > 0$ and $\eta_{0} : = \min ( \min_{\uuu^{( 0 )}} \eta_{\Gamma} , 1 ) / 2 > 0$. The previous paragraphs yield
\begin{equation*}
\sigma ( P_{M_{\Gamma}} ) \cap \big] - \infty , 2 \eta_{0} h^{2} \big] =
\left\{\begin{aligned}
&\{ \lambda_{\Gamma} ( h ) \} \quad &&\text{if } \Gamma \in \mathcal{U}^{( 0 )} ,   \\
&\emptyset &&\text{if } \Gamma \in \uuu \setminus \uuu^{( 0 )} ,
\end{aligned}\right.
\end{equation*}
for all $h \in ] 0 , h_{0} ]$. To conclude the proof it suffices to apply Theorem 2.4 of \cite{HeSj84_01} (see also Remark 2.4 of \cite{HeSj87_01}) which states that for sufficiently small $h$, the spectrum of $\Delta_{f}$ in $] - \infty , \eta_{0} h^{2} ]$ counting multiplicities is exponentially close to the union of the spectra of $P_{M_{\Gamma}}$. This ends the proof of Theorem \ref{e14}.

\section{Geometrical study near the critical manifolds}\label{s3}

In this section we prove Proposition \ref{a23} together with topological results on separating saddle manifolds needed in our construction of the quasimodes. We start with the following elementary result.

\begin{lemma}\sl \label{a24}
Let $\varphi$ be a local diffeomorphism of $\R^{d}$ defined in a neighborhood of $a\in \mathbb R^d$. Then, there exists $r_{a} > 0$ such that, for all $b \in B ( a , r_{a} )$ and $0 < r < r_{a}$, the set $\varphi ( B (b , r ) )$ is star-shaped with respect to $\varphi ( b)$.
\end{lemma}
\begin{proof}
We have to show that, for all $b \in B ( a , r_{a} ) $, $x \in B ( 0 , r )$ and $t \in [ 0 , 1 ]$, the point $( 1 - t ) \varphi ( b ) + t \varphi ( b + x )$ belongs to $\varphi ( B ( b , r ) )$. In other words,
\begin{equation*}
\forall t \in [ 0 , 1 ] , \qquad g( t ) := \big\vert \varphi^{- 1} \big( \varphi ( b ) + t ( \varphi ( b + x ) - \varphi ( b ) ) \big) - b \big\vert^{2} \in [ 0 , r^{2} [ .
\end{equation*}
On one hand, $g( 0 ) = 0$ and $g ( 1 ) = \vert x \vert^{2} < r^{2}$. On the other hand, the Taylor formula implies
\begin{align*}
\partial_{t} g ( t ) &= 2 \big\< \partial_{t} \varphi^{- 1} \big( \varphi ( b ) + t ( \varphi ( b + x ) - \varphi ( b ) ) \big) , \varphi^{- 1} \big( \varphi ( b ) + t ( \varphi ( b + x ) - \varphi ( b ) ) \big) - b \big\>   \\
&= 2 \big\< d_{\varphi ( b ) + t ( \varphi ( b + x ) - \varphi ( b ) )} \varphi^{- 1} ( \varphi ( b + x ) - \varphi ( b ) ) , \varphi^{- 1} \big( \varphi ( b ) + t ( \varphi ( b + x ) - \varphi ( b ) ) \big) - b \big\>    \\
&= 2 \big\< d_{\varphi ( b ) + \ooo ( t x )} \varphi^{- 1} ( d_{b} \varphi ( x ) + \ooo ( x^{2} ) ) , \varphi^{- 1} \big( \varphi ( b ) + t d_{b} \varphi ( x ) + \ooo ( t x^{2} ) \big) - b \big\>   \\
&= 2 \big\< d_{\varphi ( b )} \varphi^{- 1} d_{b} \varphi ( x ) + \ooo ( x^{2} ) , t d_{\varphi ( b )} \varphi^{- 1} d_{b} \varphi ( x ) + \ooo ( t x^{2} ) \big\>  \\
&= 2 t x^{2} + \ooo ( t x^{3} ) .
\end{align*}
Thus, for $r_{a} > 0$ small enough, we get $\partial_{t} g( t ) \geq 0$. Summing up, $g$ is non-decreasing and $g ( t ) \in [ g ( 0 ) , g ( 1 ) ] \subset [ 0 , r^{2} [$ for all $t \in [ 0 , 1 ]$.
\end{proof}

\begin{proof}[Proof of Proposition \ref{a23}]
For simplicity, we assume that $\sigma=0$. Let us first consider the case $\Gamma \in \uuu^{( j )}$ with $j \geq 2$. Under Assumption \ref{hyp1}, for all $a \in \Gamma$, there exists a diffeomorphism $\varphi$ of $\R^{d}$ from a neighborhood of $a$ to a neighborhood of $0$ such that $f \circ \varphi^{- 1}$ takes the form \eqref{a3}. Then, $\varphi ( X_{\sigma} )$ writes $\{ - y_{-}^{2} + y_{+}^{2} < 0 \}$ near $y = 0$. Let $r_{a} > 0$ be given by Lemma \ref{a24}. We now prove that
\begin{equation} \label{a25}
\forall b \in \Gamma \cap B ( a , r_{a} ) , \ \forall 0 < r < r_{a} , \qquad X_{\sigma} \cap B ( b , r ) \text{ is connected.}
\end{equation}
On one hand, consider $z \in \varphi ( X_{\sigma} \cap B ( b , r ) )$. From Lemma \ref{a24}, the expression of $\varphi ( X_{\sigma} )$ and $\varphi ( b ) \in \{ y_{-} = 0 \text{ and } y_{+} = 0 \}$, the segment $] \varphi ( b ) , z ]$ is included in $\varphi ( X_{\sigma} \cap B ( b , r ) )$. On the other hand, since $j \geq 2$, there are at least two variables $y_{-}$ and the set $\{ - y_{-}^{2} + y_{+}^{2} < 0 \}$ has a connected neighborhood of $\varphi ( b )$. These two arguments imply that $\varphi ( X_{\sigma} \cap B ( b , r ) )$ is connected and eventually \eqref{a25} holds.

Since $\Gamma$ is compact and $\Gamma \subset \cup_{a \in \Gamma} B ( a , r_{a} )$, there exists a finite number of $a_{j} \in \Gamma$ such that $\Gamma \subset \cup_{j} B ( a_{j} , r_{a_{j}} )$. Setting $\widetilde{r} = \min_{j} r_{a_{j}}$, \eqref{a25} gives
\begin{equation} \label{a26}
\forall b \in \Gamma , \ \forall 0 < r < \widetilde{r} , \qquad X_{\sigma} \cap B ( b , r ) \text{ is connected.}
\end{equation}
Let $a_{0} \in \Gamma$, $0 < r < \widetilde{r}$ and $A$ be the connected component of $X_{\sigma} \cap ( \Gamma + B ( 0 , r ) )$ containing $X_{\sigma} \cap B ( a_{0} , r )$. We also define $B = \overline{A} \cap \Gamma$. If $b \in B$, then $A$ intersects $X_{\sigma} \cap B ( b , r )$ and eventually $X_{\sigma} \cap B ( b , r ) \subset A$ from \eqref{a26}. Thus, the expression of $\varphi ( X_{\sigma} )$ gives that $B$ is a neighborhood of $b$ in $\Gamma$ showing that $B$ is open in $\Gamma$. Since $B$ is also closed (in $\Gamma$), $a_{0} \in B \neq  \emptyset$ and $\Gamma$ is connected, we obtain
\begin{equation} \label{a30}
B = \Gamma .
\end{equation}
Then, the argument below \eqref{a26} yields that $X_{\sigma} \cap B ( b , r ) \subset A$ for all $b \in \Gamma$. In other words, $A = X_{\sigma} \cap ( \Gamma + B ( 0 , r ) )$ which is connected and $i)$ follows.

Assume now that $\Gamma \in \uuu^{( 1 )}$. As before, for all $a \in \Gamma$, there exists a local diffeomorphism $\varphi$ near $a$ such that $f \circ \varphi^{- 1}$ takes the form \eqref{a3}, and $\varphi ( X_{\sigma} )$ writes $\{ - y_{-}^{2} + y_{+}^{2} < 0 \}$ near $y = 0$. Let $r_{a} > 0$ be given by Lemma \ref{a24}. We have
\begin{equation} \label{a27}
\begin{gathered}
\forall b \in \Gamma \cap B ( a , r_{a} ) , \ \forall 0 < r < r_{a} , \qquad X_{\sigma} \cap B ( b , r ) \text{ has two connected} \\
\text{components } A_{\pm} ( b , r ) = \varphi^{- 1} \big( \{ - y_{-}^{2} + y_{+}^{2} < 0 \} \cap \{ \pm y_{-} > 0 \} \big) \cap B ( b , r ) .
\end{gathered}
\end{equation}
The proof of \eqref{a27} is similar to that of \eqref{a25}, the difference is that there is now only one variable $y_{-}$. Moreover, the expression of $A_{\pm}$ gives that
\begin{equation} \label{a28}
\forall b \in \Gamma \cap B ( a , r_{a} ) , \ \forall 0 < r < r_{a} , \qquad \overline{A_{\pm} ( b , r )} \cap \Gamma \text{ is a neighborhood of } b \text{ in } \Gamma .
\end{equation}

As above \eqref{a26}, there exists a finite number of $a_{j} \in \Gamma$ such that $\Gamma \subset \cup_{j} B ( a_{j} , r_{a_{j}} )$. Noting $\widetilde{r} = \min_{j} r_{a_{j}}$, \eqref{a27} gives
\begin{equation} \label{a31}
\begin{gathered}
\forall b \in \Gamma , \ \forall 0 < r < \widetilde{r} , \qquad X_{\sigma} \cap B ( b , r ) \text{ has two connected} \\
\text{components } A_{\pm} ( b , r ) = \varphi^{- 1} \big( \{ - y_{-}^{2} + y_{+}^{2} < 0 \} \cap \{ \pm y_{-} > 0 \} \big) \cap B ( b , r ) .
\end{gathered}
\end{equation}
Let $a_{0} \in \Gamma$, $0 < r < \widetilde{r}$ and $A_{\pm} ( r )$ be the connected component of $X_{\sigma} \cap ( \Gamma + B ( 0 , r ) )$ containing $A_{\pm} ( a_{0} , r )$. Following the proof of \eqref{a30} and using \eqref{a28} and \eqref{a31}, we get
\begin{equation} \label{a29}
\Gamma \subset \overline{A_{+} ( r )} \cap \overline{A_{-} ( r )} .
\end{equation}
We now show that
\begin{equation} \label{a32}
X_{\sigma} \cap ( \Gamma + B ( 0 , r ) ) = A_{+} ( r ) \cup A_{-} ( r ) .
\end{equation}
By definition $A_{\pm} ( r ) \subset X_{\sigma} \cap ( \Gamma + B ( 0 , r ) )$. On the other hand, let $z \in X_{\sigma} \cap ( \Gamma + B ( 0 , r ) )$. There exists $b \in \Gamma$ such that $z \in X_{\sigma} \cap B ( b , r )$. By \eqref{a31}, there exists a sign $\varepsilon \in \{ + , - \}$ such that $z \in A_{\varepsilon} ( b , r )$. Let $C$ be the connected component of $X_{\sigma} \cap ( \Gamma + B ( 0 , r ) )$ containing $A_{\varepsilon} ( b , r )$. As in \eqref{a29}, $a_{0} \in \Gamma \subset \overline{C}$. From \eqref{a31}, $C$ intersects $A_{\delta} ( a_{0} , r )$ for some $\delta \in \{ + , - \}$. Since this last set is connected, $A_{\delta} ( a_{0} , r ) \subset C$, $C = A_{\delta} ( r )$ and finally $z \in A_{+} ( r ) \cup A_{-} ( r )$. This concludes the proof of \eqref{a32}.

From \eqref{a32}, we know that $X_{\sigma} \cap ( \Gamma + B ( 0 , r ) )$ has at most two connected components $A_{+} ( r )$ and $A_{-} ( r )$ for $0 < r < \widetilde{r}$. Note that, if $A_{+} ( r_{0} ) = A_{-} ( r_{0} )$ for some $0 < r_{0} < \widetilde{r}$, then $A_{+} ( r ) = A_{-} ( r )$ for all $r_{0} < r < \widetilde{r}$ since $\emptyset \neq A_{+} ( r_{0} ) \subset A_{\pm} ( r )$. Thus, either $A_{+} ( r ) = A_{-} ( r )$ for all $r > 0$ small enough or $A_{+} ( r ) \neq A_{-} ( r )$ for all $r > 0$ small enough. Combined with \eqref{a29}, this shows $ii)$.
\end{proof}

The next result is used in the construction of quasimodes.

\begin{prop}\sl \label{a1}
Let $\Gamma \in \uuu^{( 1 )}_{\rm sep}$ and $\mu ( x )$ be the unique negative eigenvalue of $\hess f ( x )$ for $x \in \Gamma$. There exist $r > 0$ and a smooth map 
\begin{equation*}
\nu : \Gamma \longrightarrow \S^{d - 1} ,
\end{equation*}
such that

$i)$ for all $x \in \Gamma$, $\nu ( x ) \in \ker ( \hess f ( x ) - \mu ( x ) )$.

$ii)$ for all $s \in ] 0 , r ]$ and $x \in \Gamma$, $x \pm s \nu ( x ) \in B_{\pm}$.
\end{prop}

In some sense, this result says that any manifold $\Gamma \in \uuu^{( 1 )}_{\rm sep}$ is ``negatively orientable'': even if $\Gamma $ is non-orientable, it admits a global smooth normal vector field of eigenvectors of $\hess f$ associated to its negative eigenvalue. To construct a non-orientable element of $\uuu^{( 1 )}_{\rm sep}$, one can consider a non-orientable smooth boundaryless compact submanifold $M \subset \R^{d - 1}$ and $f ( x ) = - x_{1}^{2} + \dist ( ( x_{2} , \ldots , x_{n} ) , M )^{2}$ near $\Gamma = \{ 0 \} \times M$. One can also deduce from the proof of Proposition \ref{a1} that there is only one $\nu ( x ) \in \S^{d}$ satisfying $i)$ and $ii)$. This result may also hold for locally separating manifolds, but is only needed in the sequel for elements of $\uuu^{( 1 )}_{\rm sep}$.

\begin{proof}[Proof of Proposition \ref{a1}]
Since $\Gamma \in \uuu^{( 1 )}$, $\hess f ( x )$ has a unique negative eigenvalue $\mu ( x )$ and $\ker ( \hess f ( x ) - \mu ( x ) )$ is a one dimensional vector space for all $x \in \Gamma$. Let $\widetilde{\nu} ( x )$ be (any) normalized element of $\ker ( \hess f ( x ) - \mu ( x ) )$. By Taylor's formula,
\begin{align*}
f ( x + s \widetilde{\nu} ( x ) ) &= f ( \Gamma ) + \frac{s^{2}}{2} \big\< \hess f ( x ) \widetilde{\nu} ( x ) , \widetilde{\nu} ( x ) \big\> + \ooo ( s^{3} ) \\
&= f ( \Gamma ) + s^{2} \mu ( x ) / 2  + \ooo ( s^{3} ),
\end{align*}
where the $\ooo ( s^{3} )$ is uniform with respect to $x$ since $\Gamma$ is compact. Using again the compactness of $\Gamma$, there exists $r > 0$ such that 
\begin{equation} \label{a5}
f ( x + s \widetilde{\nu} ( x ) ) < f ( \Gamma ) ,
\end{equation}
for all $x \in \Gamma$ and $s \in [ - r , r ]\setminus \{ 0 \}$. From Lemma \ref{a4}, it yields $x + s \widetilde{\nu} ( x ) \in B_{-} \cup B_{+}$. Since $B_{\pm}$ are connected components of $X_{f ( \Gamma )}$, $x + s \widetilde{\nu} ( x )$ stays in the same $B_{\bullet}$ for all $s \in [ - r , 0 [$ and all $s \in ] 0 , r ]$. Moreover, since $f$ takes the form \eqref{a3} near $\Gamma$, $x + s \widetilde{\nu} ( x ) \in B_{+}$ if and only if $x - s \widetilde{\nu} ( x ) \in B_{-}$. Summing up, we can choose $\nu ( x )$ equal to $\pm \widetilde{\nu} ( x )$ for all $x \in \Gamma$ such that $x + s \nu ( x ) \in B_{+}$ and $x - s \nu ( x ) \in B_{-}$ for $s \in ] 0 , r ]$. Eventually, $\nu ( x )$ is $C^{\infty}$ since $\ker ( \hess f ( x ) - \mu ( x ) )$ depends smoothly of $x \in \Gamma$ and the choice of the sign for $\nu ( x )$ respects this regularity by connectedness.
\end{proof}

\section{Local construction of the quasimodes} \label{s4}

In this part, we construct Gaussian quasimodes near the separating saddle manifolds. Such constructions go back to \cite{BoEcGaKl04,DiLe17_01,LePMi20}. Given a separating saddle manifold $\Gamma \in \uuu^{( 1 )}_{\rm sep}$, we look for a solution of the equation $\Delta_{f} u = 0$ in the neighborhood of $\Gamma$ under the form $u ( x ) = v ( x , h ) e^{- f ( x ) / h}$ with
\begin{equation} \label{a6}
v ( x , h ) = \int_{0}^{\ell ( x , h )} \zeta ( s / \tau ) e^{- s^{2} / 2 h} d s ,
\end{equation}
for some function $\ell ( x , h ) \in C^{\infty} ( \R^{d} )$ having a classical expansion $\ell ( x , h ) \sim \sum_{j \geq 0} h^{j} \ell_{j} ( x )$. Here $\zeta$ denotes a fixed smooth even function equal to $1$ on $[ - 1 , 1 ]$ and supported in $[ - 2 , 2 ]$ and $\tau > 0$ is a small parameter which will be fixed later. The object of this section is to construct the function $\ell$.

The following lemma holds by a straightforward computation.

\begin{lemma}[Equations on $\ell$]\sl \label{a7}
We have 
\begin{equation*}
\Delta_{f} ( v e^{- f / h} ) = ( w + r ) e^{- ( f + \frac{\ell^{2}}{2} ) / h} ,
\end{equation*}
where 
\begin{equation*}
w = h \big( 2 \nabla f \cdot \nabla \ell + \vert \nabla \ell \vert^{2} \ell \big) - h^{2} \Delta \ell ,
\end{equation*}
$\supp ( r ) \subset \{ \vert \ell \vert \geq \tau \}$ and $r$ and all its derivatives are uniformly bounded with respect to $h$. Moreover, $w$ admits a classical expansion $w \sim \sum_{j \geq 1} h^{j} w_{j}$ with 
\begin{equation*}
w_{1} = 2 \nabla f \cdot \nabla \ell_{0} + \vert \nabla \ell_{0} \vert^{2} \ell_{0} ,
\end{equation*}
and, for all $j \geq 1$,
\begin{equation*}
w_{j + 1} = 2 \nabla f \cdot \nabla \ell_{j} + 2 \ell_{0} \nabla \ell_{0} \cdot \nabla \ell_{j} + \vert \nabla \ell_{0} \vert^{2} \ell_{j} + R_{j} ( x , \ell_{0} , \ldots , \ell_{j - 1} ) ,
\end{equation*}
with $R_{j} \in C^{\infty} ( \R^{d + j} )$.
\end{lemma}

As usual in the study of the tunneling effect, we consider 
\begin{equation} \label{a9}
q ( x , \xi ) := - p(x,i\xi) =\xi^{2} - \vert \nabla f ( x ) \vert^{2} ,
\end{equation}
the complexified version of the principal symbol $p(x,\xi) := \xi^2 + \vert \nabla f(x)\vert^2$ of $\Delta_{f}$. For $\Gamma \in \uuu$, the set $\Gamma \times \{ 0 \} \subset T^{*} \R^{d}$ is a manifold of fixed points of the Hamiltonian vector field $H_{q} = \partial_\xi q \cdot \partial_x - \partial_x q \cdot \partial_\xi$. On the cotangent space $T^{*} \Gamma$, $q = \xi^{2}$ and the linearization of $H_{q}$ is the nilpotent matrix
\begin{equation*}
\left( \begin{array}{cc}
0 & 2 \\
0 & 0
\end{array} \right) .
\end{equation*}
On the other hand, Assumption \ref{hyp1} implies that $H_{q}$ is hyperbolic in the normal directions to $T^{*} \Gamma$. Then,  Theorem 1 in Appendix C of \cite{AbRo67_01} provides the existence of the incoming/outgoing manifolds $\Lambda_{\pm}$. They are $d$-dimensional, stable by the Hamiltonian flow and characterized near $\Gamma\times \{ 0 \}$ by
\begin{equation}
\Lambda_{\pm} = \big\{ ( x , \xi ) ; \ \dist \big( \exp ( t H_{q} ) ( x , \xi ) , \Gamma \times \{ 0 \} \big) \to 0 \text{ as } t \to \mp \infty \big\} .
\end{equation}
For $x \in \Gamma$, the tangeant space $T_{( x , 0 )} \Lambda_{+}$ (resp. $T_{( x , 0 )} \Lambda_{-}$) is spanned by the eigenvectors of
\begin{equation*}
\left( \begin{array}{cc}
0 & 2 \\
2 (\hess f )^{2} ( x ) & 0
\end{array} \right) ,
\end{equation*}
associated to the non-negative (resp. non-positive) eigenvalues. In particular, they project nicely on the base space and, as in Lemma 3.3 of \cite{DiSj99_01}, they are Lagrangian manifolds. Thus, there exist smooth functions $\phi_{\pm}$ defined near $\Gamma$ such that
\begin{equation}
\Lambda_{\pm} = \{ ( x , \nabla \phi_{\pm} ( x ) ) ; \ x \in \R^{d} \} .
\end{equation}
In \cite{HeSj86_01,HeSj87_01}, Helffer and Sj\"{o}strand identified the phase $\phi_{+}$ with the Agmon distance to $\Gamma$ defined by \eqref{e15}.

\begin{lemma}\sl \label{a8}
There exists a smooth function $\ell_{0}$ defined in a neighbohood of $\Gamma$ such that
\begin{equation*}
\phi_{+} ( x ) = f ( x ) - f ( \Gamma ) + \frac{\ell_{0}^{2} ( x )}{2} ,
\end{equation*}
\end{lemma}

\begin{proof}
This result, based on an observation of \cite[(11.20)]{HeHiSj08-2}, is similar to Lemma 3.2 of \cite{BoLePMi22}. We give its proof for a sake of completeness and to explain why this construction can be made globally around $\Gamma$. Let $\Lambda_{f} : = \{ ( x , \nabla f ( x ) ) ; \ x \in \R^{d} \}\subset T^{*} \R^{d}$ be the Lagrangian manifold associated to the function $f$. From \eqref{a9}, we have $q ( x , \nabla f ( x ) ) = 0$ which implies that $\Lambda_{f}$ is stable by the $H_{q}$ flow. Let $\fff$ denote the Hamiltonian vector field $H_{q}$ restricted to $\Lambda_{f}$. Then, $\Gamma \times \{ 0 \}$ is a manifold of fixed points for $\fff$. Moreover, the linearization of $\fff$ at $( x , 0 )$ with $x \in \Gamma$ is
\begin{equation*}
F_{x} = \left( \begin{array}{cc}
0 & 2 \\
2 ( \hess f )^{2} ( x ) & 0
\end{array} \right) ,
\end{equation*}
on the tangent space of $\Lambda_{f}$ at $( x , 0 )$
\begin{equation*}
T_{( x , 0 )} \Lambda_{f} = \{ ( t , \hess f ( x ) t ) ; \ t \in \R^{d} \}.
\end{equation*}
Since $F_{x}$ acts as $2 \hess f ( x )$, this operator has $d_{\Gamma}$ zero eigenvalues (corresponding to the tangent space of $\Gamma \times \{ 0 \}$), $1$ negative eigenvalue and $d - d_{\Gamma} - 1$ positive eigenvalues.

\begin{figure}
\begin{center}
\begin{picture}(0,0)%
\includegraphics{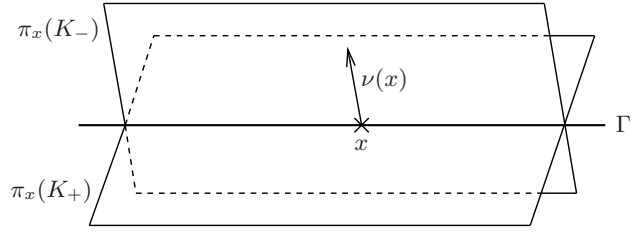}%
\end{picture}%
\setlength{\unitlength}{1184sp}%
\begingroup\makeatletter\ifx\SetFigFont\undefined%
\gdef\SetFigFont#1#2#3#4#5{%
  \reset@font\fontsize{#1}{#2pt}%
  \fontfamily{#3}\fontseries{#4}\fontshape{#5}%
  \selectfont}%
\fi\endgroup%
\begin{picture}(12705,4718)(-6764,-7296)
\put(5926,-5311){\makebox(0,0)[lb]{\smash{{\SetFigFont{9}{10.8}{\rmdefault}{\mddefault}{\updefault}$\Gamma$}}}}
\put(-6599,-3286){\makebox(0,0)[lb]{\smash{{\SetFigFont{9}{10.8}{\rmdefault}{\mddefault}{\updefault}$\pi_{x} ( K_{-} )$}}}}
\put(-6749,-6661){\makebox(0,0)[lb]{\smash{{\SetFigFont{9}{10.8}{\rmdefault}{\mddefault}{\updefault}$\pi_{x} ( K_{+} )$}}}}
\put(601,-5686){\makebox(0,0)[b]{\smash{{\SetFigFont{9}{10.8}{\rmdefault}{\mddefault}{\updefault}$x$}}}}
\put(601,-4336){\makebox(0,0)[lb]{\smash{{\SetFigFont{9}{10.8}{\rmdefault}{\mddefault}{\updefault}$\nu ( x )$}}}}
\end{picture}%
\end{center}
\caption{The geometry in the proof of Lemma \ref{a8}.}
\label{f2}
\end{figure}

Let $K_{\pm}$ be the stable outgoing/incoming submanifold of $\Lambda_{f}$ associated to $\fff$. Then $K_{+}$ (resp. $K_{-}$) has dimension $d - 1$ (resp. $d_{\Gamma} + 1$), $K_{\pm}$ projects nicely on the $x$-space,
\begin{equation} \label{a11}
K_{\pm} = \Lambda_{\pm} \cap \Lambda_{f} \qquad \text{and} \qquad T_{( x , 0 )} K_{\pm} = T_{( x , 0 )} \Lambda_{\pm} \cap T_{( x , 0 )} \Lambda_{f} ,
\end{equation}
for all $x \in \Gamma$. The existence and uniqueness of $K_{\pm}$ follows again from Theorem 1 in Appendix C of \cite{AbRo67_01}. Note that since $\fff$ is normally hyperbolic at $\Gamma \times \{ 0 \}$ (which was not the case for $H_{q}$ on the whole space when $d_{\Gamma} > 0$), we could have used instead the classical result of Hirsch, Pugh and Shub \cite{HiPuSh77_01}. The first equality in \eqref{a11} gives $\nabla \phi_{\pm} = \nabla f$ on $\pi_{x} ( K_{\pm} )$ and then
\begin{equation} \label{a10}
\forall x \in \pi_{x} ( K_{\pm} ) , \qquad \phi_{\pm} ( x ) = f ( x ) - f ( \Gamma ) .
\end{equation}
On the other hand, for all $x \in \Gamma$, we have
\begin{equation*}
\hess ( \phi_{+} - f ) ( x ) = \vert \hess f ( x ) \vert - \hess f ( x ) > 0 ,
\end{equation*}
on $\nu ( x ) \R$, where the vector field $\nu ( x )$ is given by Proposition \ref{a1}. Summing up, $\phi_{+} - f + f ( \Gamma )$ is a non-negative function in a neighborhood of $\Gamma$ which vanishes at order $2$ on $\pi_{x} ( K_{+} )$ (see Figure \ref{f2}).

We now construct a square root of $g = \phi_{+} - f + f ( \Gamma )$ in a neighborhood of $\Gamma$. Let $x$ be a point of $\Gamma$. There exist local coordinates $( y , z ) \in \R^{d - 1} \times \R$ mapping $x$ to $0$ such that $\pi_{x} ( K_{+} ) = \{ ( y , z ) ; \ z = 0 \}$ and such that the last basis vector $( 0 , \ldots , 0 , 1 )$ corresponds to $\nu ( x )$. Near $0$, we have $g ( y , 0 ) = 0$ from \eqref{a10}, $\partial_{z} g ( y , 0 ) = 0$ from \eqref{a11} and $\partial^{2}_{z , z} g ( y , z ) > 0$ from the last sentence of the previous paragraph. Then, the Taylor formula gives
\begin{equation*}
g ( y , z ) = \int_{0}^{1} ( 1 - t ) \partial^{2}_{z , z} g ( y , t z ) \, d t \, z^{2} ,
\end{equation*}
and
\begin{equation} \label{a15}
\phi_{+} = f - f ( \Gamma ) + \frac{\ell_{0 , x}^{2}}{2} \qquad \text{with} \qquad \ell_{0 , x} ( y , z ) = \Big( 2 \int_{0}^{1} ( 1 - t ) \partial^{2}_{z , z} g ( y , t z ) \, d t \Big)^{1 / 2} z.
\end{equation}
Since the quantity under the square root is positive when evaluated in $z = 0$, the function $\ell_{0 , x}$ is smooth in a vicinity of $0$. Coming back to the original variables, we have construct a smooth square root of $g$ in a neighborhood of $x$ which is positive (resp. negative) in the direction $\nu ( x )$ (resp. $- \nu ( x )$). Since $\nu$ is globally defined on the compact manifold $\Gamma$, these local functions glue together and provide a smooth function $\ell_{0}$, defined in a neighborhood of $\Gamma$, which satisfies $\phi_{+} = f - f ( \Gamma ) + \ell_{0}^{2} / 2$.
\end{proof}

\begin{lemma}[Eikonal equation]\sl \label{a12}
For $\Gamma \in \uuu^{( 1 )}_{\rm sep}$, the function $\ell_{0}$ of Lemma \ref{a8} solves
\begin{equation} \label{a13}
2 \nabla f \cdot \nabla \ell_{0} + \vert \nabla \ell_{0} \vert^{2} \ell_{0} = 0 ,
\end{equation}
in a neighborhood of $\Gamma$. Moreover, for all $x \in \Gamma$, the vector $\nabla \ell_{0} ( x )$ is an eigenvector of the matrix $\hess f ( x )$ associated to its unique negative eigenvalue $\mu ( x )$. Eventually,
\begin{equation} \label{a14}
\vert \nabla \ell_{0} ( x ) \vert^{2} = - 2 \mu ( x ) \qquad \text{and} \qquad \det \hess_{\perp} \Big( f + \frac{1}{2} \ell_{0}^{2} \Big) ( x ) = - \det \hess_{\perp} ( f ) ( x ) .
\end{equation}
\end{lemma}

\begin{proof}
Combining $0 = q ( x , \nabla \phi_{+} ) = \vert \nabla \phi_{+} \vert^{2} - \vert \nabla f \vert^{2}$ with Lemma \ref{a8} leads to
\begin{align*}
0 &= \vert \nabla f + \ell_{0} \nabla \ell_{0} \vert^{2} - \vert \nabla f \vert^{2} \\
&= \vert \nabla f \vert^{2} + 2 \ell_{0} \nabla f \cdot \nabla \ell_{0} + \ell_{0}^{2} \vert \nabla \ell_{0} \vert^{2} - \vert \nabla f \vert^{2}  \\
&= \ell_{0} \big( 2 \nabla f \cdot \nabla \ell_{0} + \vert \nabla \ell_{0} \vert^{2} \ell_{0} \big) .
\end{align*}
Since $\ell_{0}$ does not vanish outside the hypersurface $\pi_{x} ( K_{+} )$ from \eqref{a15}, it implies \eqref{a13}.

Consider now $x \in \Gamma$ and $y \in \R^{d}$. Since $\nabla f ( x + y ) = \hess f ( x ) y + \ooo ( y^{2} )$, $\ell_{0} ( x + y ) = \nabla \ell_{0} ( x ) y + \ooo ( y^{2} )$ and $\nabla \ell_{0} ( x + y ) = \nabla \ell_{0} ( x ) + \ooo ( y )$ as $y \to 0$, \eqref{a13} gives
\begin{equation*}
\big\< 2 \hess f ( x ) \nabla \ell_{0} ( x ) , y \big\> + \big\< \vert \nabla \ell_{0} ( x ) \vert^{2} \nabla \ell_{0} ( x ) , y \big\> = \ooo ( y^{2} ).
\end{equation*}
and then
\begin{equation}
2 \hess f ( x ) \nabla \ell_{0} ( x ) = - \vert \nabla \ell_{0} ( x ) \vert^{2} \nabla \ell_{0} ( x ) .
\end{equation}
On the other hand, \eqref{a15} implies that $\nabla \ell_{0} ( x ) \neq 0$. Thus, $\nabla \ell_{0} ( x )$ is an eigenvector of $\hess f ( x )$ associated to the negative eigenvalue $- \vert \nabla \ell_{0} ( x ) \vert^{2} / 2$. Since $\mu ( x )$ is the unique negative eigenvalue of $\hess f ( x )$, the first part of \eqref{a14} follows.

Eventually, the previous discussion and $\ell_{0} ( x ) = 0$ yield
\begin{equation*}
\hess \Big( f + \frac{1}{2} \ell_{0}^{2} \Big) ( x ) = \hess ( f ) ( x ) - 2 \mu ( x ) \Pi_{x} ,
\end{equation*}
with the rank-one orthogonal projection $\Pi_{x} = - ( 2 \mu ( x ) )^{- 1} \nabla \ell_{0} ( x ) \< \nabla \ell_{0} ( x ) , \cdot \>$. Since $\Pi_{x}$ is the spectral projection of $\hess f ( x )$ associated to its negative eigenvalue $\mu ( x )$, $\hess ( f + \ell_{0}^{2} / 2 ) ( x )$ has the same eigenvalues than $\hess f ( x )$ except that $\mu ( x )$ is replaced by $- \mu ( x )$. Since the determinant of a matrix is the product of its eigenvalues, the last part of \eqref{a14} holds true.
\end{proof}

We now construct the other functions $\ell_{j}$ in the spirit of \cite[Lemma 3.4]{BoLePMi22}.

\begin{lemma}[Transport equations]\sl \label{a16}
There exist an open neighborhood $V$ of $\Gamma$ and $C^{\infty}$ functions $\ell_{j}$ for $j \geq 1$ such that
\begin{equation*}
2 \nabla f \cdot \nabla \ell_{j} + 2 \ell_{0} \nabla \ell_{0} \cdot \nabla \ell_{j} + \vert \nabla \ell_{0} \vert^{2} \ell_{j} = - R_{j} ( x , \ell_{0} , \ldots , \ell_{j - 1} ) ,
\end{equation*}
near $V$, where $R_{j}$ if given by Lemma \ref{a7}.
\end{lemma}

\begin{proof}
We can solve these transport equations by induction over $j$ since $R_{j}$ depends only on the previous functions $\ell_{0} , \ldots \ell_{j - 1}$. Then, it is enough to show that, for any smooth function $r$, there exists a smooth function $u$ defined near $V$ such that
\begin{equation} \label{a19}
\lll u = r ,
\end{equation}
where $\lll$ is the operator
\begin{align*}
\lll u &= 2 \nabla f \cdot \nabla u + 2 \ell_{0} \nabla \ell_{0} \cdot \nabla u + \vert \nabla \ell_{0} \vert^{2} u  \\
&= 2 \nabla \phi_{+} \cdot \nabla u + \vert \nabla \ell_{0} \vert^{2} u ,
\end{align*}
thanks to Lemma \ref{a8}. Near each point of $\Gamma$, there exists a local change of coordinates $F : \R^{d} \ni x \rightarrow ( y , z ) \in \R^{d_{\Gamma}} \times \R^{d - d_{\Gamma}}$ such that $\Gamma = \{ ( y , z ) ; \ z = 0 \}$. In these new coordinates, $\lll$ writes
\begin{equation} \label{a17}
\lll = ( d F ) 2 \nabla_{x} \phi_{+} \circ F^{- 1} \cdot \nabla_{y , z} + \vert \nabla_{x} \ell_{0} \vert^{2} \circ F^{- 1}.
\end{equation}
Since $\nabla \phi_{+}$ vanishes on $\Gamma$, we have $\nabla_{x} \phi_{+} \circ F^{- 1} ( y , 0 ) = 0$ for all $y$. Then, the Taylor formula gives
\begin{equation*}
( d F ) 2 \nabla_{x} \phi_{+} \circ F^{- 1} ( y , z ) = ( d F ) 2 \hess ( \phi_{+} ) ( d F )^{- 1} \circ F^{- 1} ( y , 0 ) \left( \begin{array}{c}
0 \\
z 
\end{array} \right) + r ( y , z ) ,
\end{equation*}
for some smooth vector $r ( y , z ) = \ooo ( z^{2} )$ near $z = 0$. Consider the matrix on $\R^{d}$
\begin{equation*}
N ( y ) = ( d F ) 2 \hess ( \phi_{+} ) ( d F )^{- 1} \circ F^{- 1} ( y , 0 ) .
\end{equation*}
Since $\ker \hess ( \phi_{+} ) ( x ) = T_{x} \Gamma$ for $x \in \Gamma$, we deduce $\ker N ( y ) = \{ ( t_{y} , 0 ) ; \ t_{y} \in \R^{d_{\Gamma}} \}$. Thus, $N ( y )$ can be written
\begin{equation*}
N ( y ) = \left( \begin{array}{cc}
0 & L ( y ) \\
0 & M ( y )
\end{array} \right) ,
\end{equation*}
for some $( d - d_{\Gamma} ) \times ( d - d_{\Gamma} )$ (resp. $d_{\Gamma} \times ( d - d_{\Gamma} )$) invertible matrix $M ( y )$ (resp. matrix $L ( y )$). The matrix $2 \hess ( \phi_{+} ) ( F^{- 1} ( y , 0 ) )$ being real diagonalizable as a symmetric matrix, it is the same for $N ( y )$ with the same eigenvalues. Thus, $M ( y )$ is real diagonalizable with positive eigenvalues (those of $2 \hess ( \phi_{+} ) ( F^{- 1} ( y , 0 ) )$). The previous discussion and \eqref{a14} show that $\lll$ can be decomposed as
\begin{equation} \label{a18}
\lll = \lll_{0} + \lll_{\rm rem} ,
\end{equation}
with the operators
\begin{equation*}
\lll_{0} = M ( y ) z \cdot \nabla_{z} - 2 \mu ( y ) \qquad \text{and} \qquad
\lll_{\rm rem} = c_{0} ( y , z ) + c_{y} ( y , z ) \nabla_{y} + c_{z} ( y , z ) \nabla_{z} ,
\end{equation*}
where $\mu ( y )$ is a shortcut for $\mu ( F^{- 1} ( y , 0 ) )$ and for some smooth functions $c_{0} ( y , z ) = \ooo ( z )$, $c_{y} ( y , z ) = \ooo ( z )$ and $c_{z} ( y , z ) = \ooo ( z^{2} )$.

To solve \eqref{a19}, we first look for a formal solution of the form formal powers in $z$ whose coefficients are smooth functions of $y$. Then, for $m \in \N$, let $\hhh_{m}$ be the set of homogeneous polynomials in $z$ of degree $m$ whose coefficients are smooth functions of $y$. The operator $\lll_{0}$ acts on $\hhh_{m}$ for all $m \in \N$. Moreover, for $y$ fixed, there exists a basis of $\R^{d - d_{\Gamma}}$ on which $M ( y )$ is diagonal with eigenvalues $\lambda_{k} ( y ) > 0$. Then, the monomials of degree $m$ form a basis of eigenvectors of $M ( y ) z \cdot \nabla_{z} - 2 \mu ( y )$ as an operator on the homogeneous polynomials in $z$ of degree $m$. Moreover, the eigenvalue associated to $z^{\alpha} = z_{1}^{\alpha_{1}} \cdots z_{d - d_{\Gamma}}^{\alpha_{d - d_{\Gamma}}}$ is $\sum_{k} \alpha_{k} \lambda_{k} ( y ) - 2 \mu ( y ) > 0$. Thus, $M ( y ) z \cdot \nabla_{z} - 2 \mu ( y )$ is invertible on the homogeneous polynomials in $z$ of degree $m$ at $y$ fixed. By continuity and compactness of $\Gamma$, this operator is invertible for all $( y , 0 ) \in \Gamma$ with a uniformly bounded smooth inverse. This implies that $\lll_{0}$ is invertible on $\hhh_{m}$. On the other hand, the properties on the $c_{\bullet}$ imply that $\lll_{\rm rem}$ sends formal powers in $z$ of degree at least $m$ with smooth coefficients in $y$ into formal powers in $z$ of degree at least $m + 1$ with smooth coefficients in $y$. Let $\widetilde{r}$ denote the formal power expansion in $z$ with smooth coefficients in $y$ of $r$. Since $\lll_{0}$ is invertible on $\hhh_{m}$ for all $m \in \N$ and $\lll_{\rm rem}$ is a lower order operator, one can construct inductively on the order of the powers of $z$ a formal solution $\widetilde{u}$ of
\begin{equation} \label{a20}
\lll \widetilde{u} = \widetilde{r} .
\end{equation}

Starting from the previous constructions, the Borel lemma provides a smooth function $\overline{u}$ defined on $\R^{d}$ such that its formal power expansion in $z$ given by the Taylor formula is precisely $\widetilde{u}$. In particular, \eqref{a20} gives
\begin{equation} \label{a21}
\lll \overline{u} = r +  \widehat{r} ( y  , z ) ,
\end{equation}
where, for all $\alpha , \beta$, $\partial_{y}^{\alpha} \partial_{z}^{\beta} \widehat{r} ( y , z ) = \ooo ( z^{\infty} )$ near $z = 0$ uniformly for $( y , 0 ) \in \Gamma$. To treat the remainder term $\widehat{r}$ and build an exact solution of \eqref{a19}, we use the characteristic method as in the proof of Proposition 3.5 of Dimassi and Sj\"{o}strand \cite{DiSj99_01}. For that, let $\varphi_{t} ( x )$ denote the flow of $\nabla \phi_{+}$, that is
\begin{equation*}
\left\{ \begin{aligned}
&\partial_{t} \varphi_{t} ( x ) = \nabla \phi_{+} ( \varphi_{t} ( x ) ) ,  \\
&\varphi_{t} ( x ) = x .
\end{aligned} \right.
\end{equation*}
Since the Hessian of $\phi_{+}$ is positive in the directions normal to $\Gamma$, the smooth function $\varphi_{t}$ satisfies the following estimates for all $\alpha \in \N^{d}$
\begin{equation*}
\vert \pi_{z} ( \varphi_{t} ( x ) ) \vert \leq C e^{- C \vert t \vert} \vert x \vert \qquad \text{and} \qquad \vert \partial_{x}^{\alpha} \varphi_{t} ( x ) \vert \leq C_{\alpha} e^{C_{\alpha} \vert t \vert} ,
\end{equation*}
for $t \leq 0$ with $C , C_{\alpha} > 0$ and $\pi_{z} ( y , z ) = z$. We then define the function
\begin{equation}
\widehat{u} ( x ) = \int_{- \infty}^{0} e^{- \int_{t}^{0} \vert \nabla \ell_{0} \vert^{2} ( \varphi_{s} ( x ) ) \, d s} \widehat{r} ( \varphi_{t} ( x ) ) \, d t .
\end{equation}
Thanks to the previous estimates and the properties of $\widehat{r}$, this expression defines a smooth function near a neighborhood $V$ of $\Gamma$ (independent of $\widehat{r}$). Moreover, it solves $\lll \widehat{u} = \widehat{r}$. Finally, $u : = \overline{u} + \widehat{u}$ is a solution of \eqref{a19}.
\end{proof}

Combining Lemma \ref{a12}, Lemma \ref{a16} and a Borel procedure in $h$, we eventually get

\begin{prop}\sl \label{a22}
For any $\Gamma \in \uuu^{( 1 )}_{\rm sep}$, there exists a smooth function $x \mapsto \ell ( x , h )$ defined in a neighborhood $V$ of $\Gamma$ such that the following holds true.

$i)$ $\ell$ admits a classical expansion $\ell ( x , h ) \sim \sum_{j} h^{j} \ell_{j} ( x )$,

$ii)$ $2 \nabla f \cdot \nabla \ell + \vert \nabla \ell \vert^{2} \ell - h \Delta \ell = \ooo ( h^{\infty} )$ uniformly with respect to $x$ in $V$,

$iii)$ the function $- \ell ( x , h )$ satisfies also $i)$ and $ii)$.
\end{prop}

Note that $- \ell$ is also a function which can be obtained following the construction of $\ell$ but replacing the vector field $\nu ( x )$ by $ - \nu ( x  )$. Depending on the global geometry of the critical set of $f$, we will choose later which function ($\ell$ or $- \ell$) will be convenient.

\section{Global construction of the quasimodes} \label{s5}

Following Section 3 of \cite{LePMi20}, we construct a global quasimode $\psi_{m}$ associated to each minimal manifold $m \in \uuu^{(0)}$. Roughly speaking, we search this function supported in a neighborhood of $E ( m )$ and decaying from $m$. Inside $E ( m )$, we choose $\psi_{m} = e^{- f / h}$. Near a regular point $x$ of the boundary $\partial E ( m )$ (that is $f ( x ) = \bsigma ( m )$ and $\nabla f ( x ) \neq 0$) or near a non-separating critical manifold $\Gamma \in \partial E ( m )$, we still take $\psi_{m} = e^{- f / h}$ since this choice is in agreement with the values of $\psi_{m}$ already known in $X_{\bsigma ( m )}$ near $x$ or $\Gamma$. Eventually, for any separating saddle manifold $\Gamma \subset \partial E ( m )$, Lemma \ref{a4} shows that $E ( m )$ is one of the two connected components $B_{\pm}$ of $X_{\bsigma ( m )}$ near $\Gamma$. We already know the value of $\psi_{m}$ in $E ( m )$ and we want $\psi_{m} = 0$ in the other connected component. To glue these two functions together, we use the constructions of the previous section near $\Gamma$.

More precisely, let $\Gamma \in \uuu^{( 1 )}_{\rm sep}$ be a separating saddle manifold with $\Gamma \subset \partial E ( m )$ and note
\begin{equation*}
\sigma : = \bsigma ( m ) = f ( \Gamma ) .
\end{equation*}
From Lemma \ref{a4}, $X_{\sigma}$ has two connected components $B_{\pm}$ with $ \Gamma \cap \overline{B_\pm}  \neq \emptyset$ and $E ( m ) $ is one of them. Modulo a change of labeling, we assume in the sequel that $E ( m ) = B_{+}$. Let $\ell_{\Gamma} ( x , h ) : V \to \mathbb R$ be the function constructed in Section \ref{s4} and positive in $E ( m )$ (see Proposition \ref{a22} $iii)$ for the choice of sign), and let $U_\Gamma$ be some small open neighborhood of $\Gamma$ such that $\overline{U_\Gamma}\subset V$. Mimicking \eqref{a6}, we set for $x \in U_{\Gamma}$
\begin{equation} \label{a34}
v_{\Gamma} ( x , h ) = C_{\Gamma}^{- 1} \int_{0}^{\ell_{\Gamma} ( x , h )} \zeta ( s / \tau ) e^{- s^{2} / 2 h} d s ,
\end{equation}
where $\zeta \in C^{\infty}_{0} ( ] - 2 , 2 [; [ 0 , 1 ] )$ is an even function equal to $1$ on $[ - 1 , 1 ]$ and with $\tau > 0$ and the renormalization constant
\begin{equation} \label{a47}
C_{\Gamma} = \int_{0}^{+ \infty} \zeta ( s / \tau ) e^{- s^{2} / 2 h} d s = \sqrt{\frac{\pi h }{2}} \big( 1 + \ooo ( e^{- c / h} ) \big) ,
\end{equation}
for some $c > 0$.

\begin{lemma}\sl \label{a37}
Let $V_{\Gamma}^{0}$ be an open set satisfying $\Gamma \subset V_{\Gamma}^{0} \Subset U_{\Gamma}$. There exist open neighborhoods $V_{\Gamma}^{\pm}$ of $\overline{B_{\pm}} \cap \overline{U_{\Gamma}} \setminus V_{\Gamma}^{0}$ such that $v_{\Gamma} = \pm 1$ in $V_{\Gamma}^{\pm}$ for all $\tau > 0$ small enough.
\end{lemma}

\begin{proof}
Since $\phi_{+}$ has a positive Hessian in the normal directions to $\Gamma$, there exists $\beta > 0$ such that $\phi_{+} \geq \beta$ on $\overline{U_{\Gamma}} \setminus V_{\Gamma}^{0}$ (after a possible shrinking of $U_{\Gamma}$). On the other hand, we have $f \leq \sigma$ on $\overline{B_{\pm}}$.  Summing up, Lemma \ref{a8} gives
\begin{equation*}
\ell_{\Gamma , 0}^{2} = 2 \phi_{+} - 2 ( f - \sigma ) \geq 2 \beta ,
\end{equation*}
in $( \overline{B_{+}} \cup \overline{B_{-}} ) \cap \overline{U_{\Gamma}} \setminus V_{\Gamma}^{0}$. Moreover, we have $\pm \ell_{\Gamma , 0} \geq 0$ in $\overline{B_{\pm}} \cap \overline{U_{\Gamma}}$ by our choice of sign for $\ell_{\Gamma}$, Proposition \ref{a1} $ii)$ and the discussion below \eqref{a15}. By a continuity argument, the previous equation gives $\pm \ell_{\Gamma} \geq \sqrt{\beta}$ in $V_{\Gamma}^{\pm}$, an open neighborhood of $\overline{B_{\pm}} \cap \overline{U_{\Gamma}} \setminus V_{\Gamma}^{0}$. For any $0 < \tau < \sqrt{\beta} / 2$, the support properties of $\zeta$ imply that $\pm [ \ell_{\Gamma} ( x ) , + \infty [$ for $x \in V_{\Gamma}^{\pm}$ does not meet the support of $\zeta ( \cdot / \tau )$ and then
\begin{equation*}
v_{\Gamma} = \left\{ \begin{aligned}
&C_{\Gamma}^{- 1} \int_{0}^{+ \infty} \zeta ( s / \tau ) e^{- s^{2} / 2 h} d s = 1 &&\text{ in } V_{\Gamma}^{+} , \\
&C_{\Gamma}^{- 1} \int_{0}^{- \infty} \zeta ( s / \tau ) e^{- s^{2} / 2 h} d s = - 1 &&\text{ in } V_{\Gamma}^{-} ,
\end{aligned} \right.
\end{equation*}
the second identity using that $\zeta$ is even.
\end{proof}

\begin{figure}
\begin{center}
\begin{picture}(0,0)%
\includegraphics{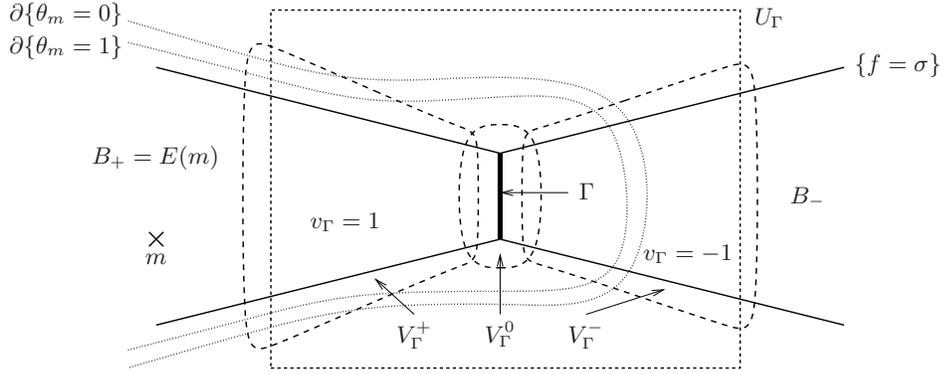}%
\end{picture}%
\setlength{\unitlength}{1184sp}%
\begingroup\makeatletter\ifx\SetFigFont\undefined%
\gdef\SetFigFont#1#2#3#4#5{%
  \reset@font\fontsize{#1}{#2pt}%
  \fontfamily{#3}\fontseries{#4}\fontshape{#5}%
  \selectfont}%
\fi\endgroup%
\begin{picture}(17730,7566)(-10289,-7894)
\put(6076,-4411){\makebox(0,0)[lb]{\smash{{\SetFigFont{9}{10.8}{\rmdefault}{\mddefault}{\updefault}$B_{-}$}}}}
\put(-7199,-5686){\makebox(0,0)[b]{\smash{{\SetFigFont{9}{10.8}{\rmdefault}{\mddefault}{\updefault}$m$}}}}
\put(-8549,-3586){\makebox(0,0)[lb]{\smash{{\SetFigFont{9}{10.8}{\rmdefault}{\mddefault}{\updefault}$B_{+} = E ( m )$}}}}
\put(7426,-1636){\makebox(0,0)[lb]{\smash{{\SetFigFont{9}{10.8}{\rmdefault}{\mddefault}{\updefault}$\{ f = \sigma \}$}}}}
\put(-1799,-7321){\makebox(0,0)[b]{\smash{{\SetFigFont{9}{10.8}{\rmdefault}{\mddefault}{\updefault}$V_{\Gamma}^{+}$}}}}
\put(  1,-7321){\makebox(0,0)[b]{\smash{{\SetFigFont{9}{10.8}{\rmdefault}{\mddefault}{\updefault}$V_{\Gamma}^{0}$}}}}
\put(-3974,-4936){\makebox(0,0)[lb]{\smash{{\SetFigFont{9}{10.8}{\rmdefault}{\mddefault}{\updefault}$v_{\Gamma} = 1$}}}}
\put(1801,-7321){\makebox(0,0)[b]{\smash{{\SetFigFont{9}{10.8}{\rmdefault}{\mddefault}{\updefault}$V_{\Gamma}^{-}$}}}}
\put(3001,-5611){\makebox(0,0)[lb]{\smash{{\SetFigFont{9}{10.8}{\rmdefault}{\mddefault}{\updefault}$v_{\Gamma} = - 1$}}}}
\put(1651,-4336){\makebox(0,0)[lb]{\smash{{\SetFigFont{9}{10.8}{\rmdefault}{\mddefault}{\updefault}$\Gamma$}}}}
\put(-10274,-1261){\makebox(0,0)[lb]{\smash{{\SetFigFont{9}{10.8}{\rmdefault}{\mddefault}{\updefault}$\partial \{ \theta_{m} = 1 \}$}}}}
\put(-10274,-586){\makebox(0,0)[lb]{\smash{{\SetFigFont{9}{10.8}{\rmdefault}{\mddefault}{\updefault}$\partial \{ \theta_{m} = 0 \}$}}}}
\put(5326,-661){\makebox(0,0)[lb]{\smash{{\SetFigFont{9}{10.8}{\rmdefault}{\mddefault}{\updefault}$U_{\Gamma}$}}}}
\end{picture}%
\end{center}
\caption{The geometric setting near a separating saddle manifold.}
\label{f3}
\end{figure}

We are now in position to define the global quasimode $\psi_{m}$. Recall (see \eqref{a53}) that ${\bf j}(m)$ denotes  the set of separating saddle manifolds $\Gamma \in \uuu^{( 1 )}_{\rm sep}$ such that $\Gamma \cap \partial E ( m ) \neq 0$ (or equivalently $\Gamma \subset \partial E ( m )$ by  \eqref{e9}). If $m \neq \underline{m}$, we have ${\bf j}(m) \neq \emptyset$. Let $\delta > 0$ be a small enough parameter fixed in the sequel. For all $\Gamma \in {\bf j}(m)$, consider $V_{\Gamma}^{0}$ as in Lemma \ref{a37} with $V_{\Gamma}^{0} \subset E ( m ) + B ( 0 , \delta )$. Let $\theta_{m} \in C^{\infty}_{0} ( E ( m ) + B ( 0 , 2 \delta ) ; [ 0 , 1 ] )$ be a function equal to $1$ near $\overline{E ( m )}$ and $\overline{V_{\Gamma}^{0}}$ for all $\Gamma \in {\bf j}(m) $ and such that
\begin{equation} \label{a38}
\supp ( \theta_{m} ) \cap \partial U_{\Gamma} \subset V_{\Gamma}^{+} ,
\end{equation}
for all $\Gamma \in {\bf j}(m)$ (see Figure \ref{f3}).

\begin{defin}\sl \label{a39}
For any $m \in \uuu^{( 0 )}$, let us define the function $\psi_{m}$ by
\begin{equation} \label{a40}
\psi_{m} : = \left\{ \begin{aligned}
&\theta_{m} ( v_{\Gamma} + 1 ) e^{- ( f - f ( m ) ) / h} &&\text{ in } U_{\Gamma} \text{ for all } \Gamma \in {\bf j}(m) , \\
&2 \theta_{m} e^{- ( f - f ( m ) ) / h} &&\text{ in } \R^{d} \setminus \bigcup_{\Gamma \in {\bf j}(m)} U_{\Gamma} ,
\end{aligned} \right.
\end{equation}
when $m \neq \underline{m}$ and by $\psi_{m} ( x ) : = e^{- ( f ( x ) - f ( m ) ) / h}$ when $m = \underline{m}$.
\end{defin}

These functions satisfy the following properties.

\begin{lemma}\sl \label{a41}
For $\delta > 0$ and then $\tau > 0$ small enough, one has

$i)$ $\psi_{m} \in C^{\infty}_{0} ( \R^{d} )$ for $m \neq \underline{m}$.

$ii)$ if $\bsigma ( m ) = \bsigma ( m^{\prime} )$ and $m \neq m^{\prime}$, then $\supp ( \psi_{m} ) \cap \supp ( \psi_{m^{\prime}} ) = \emptyset$.

$iii)$ if $\bsigma ( m ) > \bsigma ( m^{\prime} )$, then
\begin{align*}
&\star \text{ either } \supp ( \psi_{m} ) \cap \supp ( \psi_{m^{\prime}} ) = \emptyset ,  \\
&\star \text{ or } \psi_{m} = 2 e^{- ( f - f ( m ) ) / h} \text{ on } \supp ( \psi_{m^{\prime}} ) \text{ and } f ( m^{\prime} ) > f ( m ) .
\end{align*}
\end{lemma}

\begin{proof}
From \eqref{a38} and Lemma \ref{a37}, we have $v_{\Gamma} = 1$ near $\supp ( \theta_{m} ) \cap \partial U_{\Gamma}$. Thus, $\psi_{m}$ is smooth near $\partial U_{\Gamma}$ and eventually on $\R^{d}$. This proves $i)$. The two other points 
are a consequence of Assumption \ref{hyp2} and similar to Lemma 4.4 $iv)$, $v)$ of \cite{LePMi20}. We send the reader to this paper for the detailed proof.
\end{proof}

\begin{prop}\sl \label{a42}
Let Assumptions \ref{hyp0}, \ref{hyp1} and \ref{hyp2} hold and let $m \in \uuu^{( 0 )} \setminus \{ \underline{m} \}$. For $\tau_{0} , \delta_{0} > 0$ small enough, we have
\begin{equation*}
\begin{aligned}
&i) &&\Vert \psi_{m} \Vert^{2} \in \eee_{\rm cl} \Big( 4 ( \pi h )^{\frac{d - d_{m}}{2}} \int_{m} \big( \det \hess_{\perp} f ( s ) \big)^{- 1 / 2} d s \Big) , \\
&ii) &&\< \Delta_{f} \psi_{m} , \psi_{m} \> \in \sum_{\Gamma \in {\bf j}(m) } \eee_{\rm cl} \Big( \frac{4}{\pi^{2}} ( \pi h )^{\frac{d - d_{\Gamma}}{2} + 1} e^{- 2 ( \sigma - f( m ) ) / h} \int_{\Gamma} \vert \mu ( s ) \vert \big\vert \det \hess_{\perp} f ( s ) \big\vert^{- 1 / 2} d s \Big) ,  \\
&iii) &&\Vert \Delta_{f} \psi_{m} \Vert^{2} = \ooo ( h^{\infty} ) \< \Delta_{f} \psi_{m} , \psi_{m} \> ,
\end{aligned}
\end{equation*}
where $a_{h} = \eee_{\rm cl} ( b_{h} )$ means that there exists $c_{h}$ such that $a_{h} = b_{h} c_{h}$ for $h$ small enough and $c_{h}$ admits a classical expansion $c_{h} \sim \sum_{j \in \N} c_{j} h^{j}$ with $c_{0} = 1$.
\end{prop}

\begin{proof}
By Assumption \ref{hyp2}, $m$ is the unique minimal manifold of $f$ on $E ( m )$ and then on $\supp ( \theta_{m} )$ for $\delta > 0$ small enough. Then, \eqref{a40} shows that, for any small neighborhood $V$ of $m$, we have
\begin{equation*}
\Vert \psi_{m} \Vert^{2} = 4 \int_{V} e^{- 2 ( f ( x ) - f ( m ) ) / h} d x + \ooo ( e^{- c / h } ) ,
\end{equation*}
for some $c > 0$. Since the Hessian of $f$ is positive in the normal directions to $m$, we can apply a generalization of the Laplace method to the case of critical manifolds. More precisely, Hypothesis \ref{hyp1} and Theorem A.1 of \cite{Lu19_01} give
\begin{equation*}
\Vert \psi_{m} \Vert^{2} \in \eee_{\rm cl} \Big( 4 ( \pi h )^{\frac{d - d_{m}}{2}} \int_{m} \det \big( \hess_{\perp} f ( s ) \big)^{- 1 / 2} d s \Big) ,
\end{equation*}
and $i)$ follows.

Let $g_{m}$ be the smooth function equal to $\theta_{m} ( v_{\Gamma} + 1 )$ in $U_{\Gamma}$ and equal to $2 \theta_{m}$ near $\R^{d} \setminus \cup U_{\Gamma}$. Then, \eqref{NN} and \eqref{a40} give
\begin{equation*}
\< \Delta_{f} \psi_{m} , \psi_{m} \> = \Vert d_{f} \psi_{m} \Vert^{2} = \big\Vert d_{f} ( g_{m} e^{- ( f - f( m ) ) / h}) \big\Vert^{2} .
\end{equation*}
Since $d_f e^{- ( f - f( m ) ) / h} = 0$, this yields
\begin{equation} \label{a44}
\< \Delta_{f} \psi_{m} , \psi_{m} \> = h^{2} \int_{\R^{d}} \vert \nabla g_{m} ( x ) \vert^{2} e^{- 2 ( f ( x ) - f( m ) ) / h} d x .
\end{equation}
On $\R^{d} \setminus \cup U_{\Gamma}$, we have $f > \sigma$ on the support of $\nabla g_{m} = 2 \nabla \theta_{m}$ for $\delta$ small enough. On $U_{\Gamma}$, we can write
\begin{equation*}
\nabla g_{m} = ( v_{\Gamma} + 1 ) \nabla \theta_{m} + \theta_{m} \nabla v_{\Gamma} .
\end{equation*}
By the support properties of $\nabla \theta_{m}$ and Lemma \ref{a37} (see Figure \ref{f3}), we have $f > \sigma$ on $\supp ( ( v_{\Gamma} + 1 ) \nabla \theta_{m} )$. On the other hand, \eqref{a34} gives
\begin{equation*}
\theta_{m} \nabla v_{\Gamma} = C_{\Gamma}^{- 1} \theta_{m} \zeta ( \ell_{\Gamma}  / \tau ) e^{- \ell_{\Gamma}^{2} / 2 h} \nabla \ell_{\Gamma} .
\end{equation*}
Then, \eqref{a44} becomes
\begin{align}
\< \Delta_{f} \psi_{m} , \psi_{m} \> = \sum_{\Gamma \in {\bf j}(m) } C_{\Gamma}^{- 2} h^{2} \int_{U_{\Gamma}} \theta_{m}^{2} \zeta^{2} ( \ell_{\Gamma}  / \tau ) \vert \nabla \ell_{\Gamma} & \vert^{2} e^{- 2 ( f + \ell_{\Gamma}^{2} / 2 - f( m ) ) / h} d x  \nonumber \\
&\qquad \ + \ooo ( e^{- 2 ( \sigma - f( m ) ) / h - c / h} ) ,   \label{a45}
\end{align}
for some $c > 0$. Since $\ell_{\Gamma} \sim \sum_{j\geq 0} h^j \ell_{\Gamma,j}$ has an asymptotic expansion in powers of $h$, the phase factor can be decomposed as
\begin{equation}\label{pli}
e^{- 2 ( f + \ell_{\Gamma}^{2} / 2 - f( m ) ) / h} = e^{- 2 ( f + \ell_{\Gamma , 0}^{2} / 2 - f( m ) ) / h} \big( e^{- 2 \ell_{\Gamma , 0} \ell_{\Gamma , 1}} + \ooo ( h ) \big) ,
\end{equation}
where $\ell_{\Gamma , 0} = 0$ on $\Gamma$ and the remainder term $\ooo ( h )$ is a symbol. Note that $\phi_{\Gamma , +} = f + \ell_{\Gamma , 0}^{2} / 2$ from Lemma \ref{a8} where the Hessian of $\phi_{\Gamma , +}$ is positive in the normal directions to $\Gamma$. Then, we can apply the Laplace method to compute \eqref{a45}. Using \eqref{a14}, \eqref{a47}, \eqref{pli}, $\theta_{m} = \zeta ( \ell_{\Gamma}  / \tau ) = 1$, $f + \ell_{\Gamma,0}^{2} / 2 = \sigma$ on $\Gamma$ and $e^{- 2 \ell_{\Gamma , 0} \ell_{\Gamma , 1}} = 1$ on $\Gamma$ to compute all the coefficients, Theorem A.1 of \cite{Lu19_01} yields
\begin{equation*}
\< \Delta_{f} \psi_{m} , \psi_{m} \> \in \sum_{\Gamma \in {\bf j}(m) } \eee_{\rm cl} \Big( \frac{4}{\pi^{2}} ( \pi h )^{\frac{d - d_{\Gamma}}{2} + 1} e^{- 2 ( \sigma - f( m ) ) / h} \int_{\Gamma} \vert \mu ( s ) \vert \big\vert \det \hess_{\perp} f ( s ) \big\vert^{- 1 / 2} d s \Big) ,
\end{equation*}
and $ii)$ follows.

We now estimate $\Vert \Delta_{f} \psi_{m} \Vert^{2}$. Near $\R^{d} \setminus \cup U_{\Gamma}$, \eqref{a40} gives $\Delta_{f} \psi_{m} = 2 [ \Delta_{f} , \theta_{m} ] e^{- ( f - f ( m ) ) / h}$ since $\Delta_{f} e^{- ( f - f ( m ) ) / h} = 0$. Using $f > \sigma$ on $\supp ( \nabla \theta_{m} )$, we deduce
\begin{equation} \label{a51}
\Vert \Delta_{f} \psi_{m} \Vert_{L^{2} ( \R^{d} \setminus \cup U_{\Gamma} )}^{2} = \ooo \big( e^{- 2 ( \sigma - f( m ) ) / h - c / h} \big) ,
\end{equation}
for some $c > 0$. On $U_{\Gamma}$ with $\Gamma \in {\bf j}(m) $, we can write
\begin{align}
\Delta_{f} \psi_{m} &= \Delta_{f} \theta_{m} ( v_{\Gamma} + 1 ) e^{- ( f - f ( m ) ) / h}  \nonumber \\
&= \theta_{m} \Delta_{f} ( v_{\Gamma} + 1 ) e^{- ( f - f ( m ) ) / h} + [ \Delta_{f} , \theta_{m} ] ( v_{\Gamma} + 1 ) e^{- ( f - f ( m ) ) / h}  \nonumber \\
&= \theta_{m} \Delta_{f} v_{\Gamma} e^{- ( f - f ( m ) ) / h} + [ \Delta_{f} , \theta_{m} ] ( v_{\Gamma} + 1 ) e^{- ( f - f ( m ) ) / h}  \label{a46}
\end{align}
The choice of $\theta_{m}$ and Lemma \ref{a37} (see Figure \ref{f3}) imply that $f > \sigma$ on $\supp ( v_{\Gamma} + 1 ) \cap \supp ( \nabla \theta_{m} )$. Thus, the last term of \eqref{a46} satisfies
\begin{equation} \label{a48}
[ \Delta_{f} , \theta_{m} ] ( v_{\Gamma} + 1 ) e^{- ( f - f ( m ) ) / h} = \ooo \big( e^{- 2 ( \sigma - f( m ) ) / h - c / h} \big) ,
\end{equation}
for some $c > 0$. On the other hand, Lemma \ref{a7} and Proposition \ref{a22} show that
\begin{equation} \label{a49}
\theta_{m} \Delta_{f} v_{\Gamma} e^{- ( f - f ( m ) ) / h} = \theta_{m} ( w + r ) e^{- ( f + \ell_{\Gamma}^{2} / 2 - f( m ) ) / h} ,
\end{equation}
with $r = \ooo ( 1 )$, $\supp ( r ) \subset \{ \vert \ell_{\Gamma} \vert \geq \tau \}$ and $w = \ooo ( h^{\infty} )$. On $\supp ( \theta_{m} r )$, Lemma \ref{a8} gives $f + \ell_{\Gamma}^{2} / 2 = \phi_{\Gamma , +} > \sigma$ since the Hessian of $\phi_{\Gamma , +}$ is positive in the normal directions to $\Gamma$. Then, we obtain
\begin{equation*}
\big\Vert \theta_{m} r e^{- ( f + \ell_{\Gamma}^{2} / 2 - f( m ) ) / h} \big\Vert = \ooo \big( e^{- ( \sigma - f( m ) ) / h - c / h} \big) ,
\end{equation*}
for some $c > 0$. Concerning $w$, we can only deduce from Lemma \ref{a8} that $f + \ell_{\Gamma}^{2} / 2 = \phi_{\Gamma , +} \geq \sigma$ on $\supp ( \theta_{m} w )$. Since $w = \ooo ( h^{\infty} )$, we get
\begin{equation*}
\big\Vert \theta_{m} w e^{- ( f + \ell_{\Gamma}^{2} / 2 - f( m ) ) / h} \big\Vert = \ooo \big( h^{\infty} e^{- ( \sigma - f( m ) ) / h} \big) ,
\end{equation*}
Combining \eqref{a49} with the two last inequalities, it comes
\begin{equation} \label{a50}
\theta_{m} \Delta_{f} v_{\Gamma} e^{- ( f - f ( m ) ) / h} = \ooo \big( h^{\infty} e^{- ( \sigma - f( m ) ) / h} \big) .
\end{equation}
Then, \eqref{a46}, \eqref{a48} and \eqref{a50} give 
\begin{equation} \label{a52}
\Vert \Delta_{f} \psi_{m} \Vert_{L^{2} ( U_{\Gamma} )}^{2} = \ooo \big( h^{\infty} e^{- 2 ( \sigma - f( m ) ) / h} \big) .
\end{equation}
Eventually, $iii)$ follows from \eqref{a51}, \eqref{a52} and $ii)$.
\end{proof}

\section{Proof of Theorem \ref{e12}}

Recall that, for $m \in \mathcal{U}^{( 0 )}$, $S ( m ) : = \bsigma ( m ) - f ( m )$ where $\bsigma$ is defined in \eqref{e11}. From now, one labels the minimal submanifolds $m_{1} , \ldots , m_{n_{0}} \in \mathcal{U}^{( 0 )}$ of $f$ so that $( S ( m_{j} ) )_{j \in \{ 1 , \ldots , n_{0} \}}$ is non-decreasing, i.e., 
\begin{equation*}
S(m_{n_{0}}) = + \infty \qquad \text{and} \qquad \forall j \in \{ 2 , \ldots , n_{0} - 1 \} , \qquad S ( m_{j} ) \geq S ( m_{j - 1} ) .
\end{equation*}
For $j \in \{ 1, \ldots , n_{0} \}$, we set
\begin{equation}
S_{j} : = S ( m_{j} ) , \qquad \varphi_{j} : = \frac{\psi_{m_{j}}}{\Vert \psi_{m_{j}} \Vert_{L^{2}}} \qquad \text{and} \qquad \mu_{j} : = \< \Delta_{f} \varphi_{j} , \varphi_{j} \> .
\end{equation}
Let $\eta_{0} > 0$ be as in Theorem \ref{e14} and introduce the spectral projection
\begin{equation*}
\Pi_{h} : =  \frac{1}{2 i \pi} \int_{\partial D ( 0 , \frac{\eta_{0}}{2} h^{2} )} ( z - \Delta_{f} )^{- 1} d z .
\end{equation*}
For $j \in \{ 1 , \ldots , n_{0} \}$ and $h > 0$ small enough, we set $v_{j} : = \Pi_{h} \varphi_{j}$.

According to Proposition \ref{a42}, we have for any $j \in \{ 1 , \ldots , n_{0} \}$
\begin{equation}\label{e17}
\Vert \Delta_{f} \varphi_{j} \Vert = \mathcal{O} \Big( h^{\infty} \sqrt{\< \Delta_{f} \varphi_{j} , \varphi_{j} \>} \Big) = \mathcal{O} ( h^{\infty} \sqrt{\mu_{j}} ) .
\end{equation}
On the other hand, using Lemma \ref{a41} and proceeding in the same way as in the proofs of \cite[Proposition 5.1 (i)]{BoLePMi22} and \cite[Lemma 4.7]{LePMi20} respectively, we also have for any $j , k \in \{ 1 , \ldots , n_{0} \}$
\begin{equation}\label{e16}
\< \varphi_{j} , \varphi_{k} \> = \delta_{j , k} + \mathcal{O} ( e^{- c / h} ) \qquad \text{and} \qquad \< \Delta_{f} \varphi_{j} , \varphi_{k} \> = \delta_{j , k} \mu_{j} ,
\end{equation}
for some constant $c > 0$, uniformly for $h > 0$ small enough.

Writing 
\begin{equation*}
( 1 - \Pi_{h} ) \varphi_{j} = - \frac{1}{2 i \pi} \int_{\partial D ( 0 , \frac{\eta_{0}}{2} h^{2} )} z^{- 1} ( z - \Delta_{f} )^{- 1} \Delta_{f} \varphi_{j} \, d z ,
\end{equation*}
and using \eqref{e17} together with $\Vert ( z - \Delta_{f} )^{- 1} \Vert = \mathcal{O} ( h^{- 2} )$ for any $z \in \partial D ( 0 , \frac{\eta_{0}}{2} h^{2} )$, we get
\begin{equation}\label{e18}
( 1 - \Pi_{h} ) \varphi_{j} = \mathcal{O} ( h^{\infty} \sqrt{\mu_{j}} ) .
\end{equation}
Combining estimates \eqref{e17}, \eqref{e16} and \eqref{e18}, we obtain the following proposition (see for instance \cite[Proposition 4.10]{LePMi20} for the proof).

\begin{prop}\sl
There exists a constant $c > 0$ such that, for all $j , k \in \{ 1 , \ldots , n_{0} \}$,
\begin{equation}\label{e19}
\< v_{j} , v_{k} \> = \delta_{j , k} + \mathcal{O} ( e^{- c / h} ) ,
\end{equation}
and 
\begin{equation}\label{e20}
\< \Delta_{f} v_{j} , v_{k} \> = \delta_{j , k} \mu_{j} + \mathcal{O} ( h^{\infty} \sqrt{\mu_{j} \mu_{k} } ) .
\end{equation}
In particular, $\{ v_{1} , \ldots , v_{n_{0}} \}$ is a basis of $\ran \Pi_{h}$ for $h > 0$ small enough.
\end{prop}

Relying on the above result, the rest of the proof is standard. We recall only the main steps referring for instance to \cite[Proposition 4.12]{LePMi20} for the details. Starting from the basis $\{ v_{n_{0} - j + 1} \}_{j \in \{ 1 , \ldots , n_{0} \}}$ we obtain an orthonormal basis $\{ e_{n_{0} - j + 1} \}_{j \in \{ 1 , \ldots , n_{0} \}}$ of $\ran \Pi_{h}$ using the Gram--Schmidt process. Thanks to \eqref{e19}, this new basis satisfies for any $j \in \{ 1 , \ldots , n_{0} \}$,
\begin{equation*}
e_{n_{0} - j + 1} = v_{n_{0} - j + 1} + \mathcal{O} ( e^{- c / h} ) .
\end{equation*}
Now, it follows from the above labeling and \eqref{e20} that, for any $j , k \in \{1 , \ldots , n_{0} \}$,
\begin{equation*}
\< \Delta_{f} e_{n_{0} - j + 1} , e_{n_{0} - k + 1} \> = \delta_{j , k} \mu_{n_{0} - j + 1} + \mathcal{O} ( h^{\infty} \sqrt{\mu_{n_{0} - j + 1} \mu_{n_{0} - k + 1}} ) .
\end{equation*}
Hence the matrix $M_{h}$ of ${\Delta_{f}}_{\vert \ran \Pi_{h}}$ in the basis $\{ e_{n_{0} - j + 1} \}_{j \in \{ 1 , \ldots , n_{0} \}}$ takes the form 
\begin{equation*}
M_{h} = \diag \big( ( \sqrt{\mu_{n_{0} - j + 1}} )_{1 \leq j \leq n_{0}} \big) ( I_{n_{0}} + \mathcal{O} ( h^{\infty} ) ) \diag \big( ( \sqrt{\mu_{n_{0} - j + 1}} )_{1 \leq j \leq n_{0}} \big) .
\end{equation*}
The spectrum of such a matrix can be computed using the Fan inequalities (see \cite{GoKr71_01, Si79}) or Lemma \ref{e21} below. From these results, the eigenvalues $\lambda_{j} ( h )$ of $M_{h}$, that are the small eigenvalues of $\Delta_{f}$, satisfy  
\begin{equation}\label{e22}
\lambda_{j} ( h ) = \mu_{n_{0} - j + 1} ( 1 + \mathcal{O} ( h^{\infty} ) ) = \< \Delta_{f} \varphi_{n_{0} - j + 1} , \varphi_{n_{0} - j + 1} \> ( 1 + \mathcal{O} ( h^{\infty} ) ) .
\end{equation}
Eventually, the announced result follows from $i)$ and $ii)$ of Proposition \ref{a42}.

\begin{lemma}\sl \label{e21}
Let $M = M ( h ) $ be a $n \times n$ matrix of the form
\begin{equation*}
M = D ( I_n + \ooo ( h^{\infty} ) ) D ,
\end{equation*}
for some diagonal matrix $D ( h ) = \diag ( \nu _{j} ( h ) )$ with $\nu_{j} ( h ) \in \C$. Then, the eigenvalues of $M$ are of the form $\nu_{j}^{2} ( h ) ( 1 + \ooo ( h^{\infty} ) )$.
\end{lemma}

\begin{proof}
Without loss of generality, we can assume that $\nu_{j} \neq 0$ for all $j \in \{ 1 , \ldots , n \}$ (we simply remove the lines and columns of zeros if some of the $\nu_{j}$'s vanish). The eigenvalues of $M$ are the zeros of the renormalized characteristic polynomial
\begin{align}
p(z) &= \det(D)^{-2} \det ( M - z ) = \det ( D^{- 1} ( M - z ) D^{- 1} )  \nonumber \\
&= \det \big( \diag ( 1 - z \nu_{j}^{-2} ) + \ooo ( h^{\infty} ) \big) .
\end{align}
The usual expansion formula for the determinant using permutations allows to write $p ( z ) = f ( z ) + g ( z )$ where $f ( z ) = \prod_{j \in \{ 1 , \ldots , n \}} ( 1 - z \nu_{j}^{-2} )$ and $g ( z )$ is a finite sum of terms of the form $\ooo ( h^{\infty} ) \prod_{j \in J} ( 1 - z \nu_{j}^{-2} )$ with $J \varsubsetneq \{ 1 , \ldots , n \}$. For $K > 0$, we have $\vert 1 - z \nu_{j}^{-2} \vert \geq h^{K}$ for $z \notin B ( \nu_{j}^{2} , h^{K} \vert \nu_{j} \vert^{2} )$. It follows that
\begin{equation*}
\forall z \in \C \setminus \cup_{j} B ( \nu_{j}^{2} , h^{K} \vert \nu_{j} \vert^{2} ) , \qquad \vert g ( z ) \vert < \vert f ( z ) \vert ,
\end{equation*}
for $h$ small enough. Letting $K$ goes to $+ \infty$, the Rouch\'e Theorem implies the Lemma.
\end{proof}

\noindent
\textbf{Acknowledgement}\\
This work was  supported by the ANR-19-CE40-0010, Analyse Quantitative de Processus M\'etastables (QuAMProcs). The first author acknowledges the financial support of the project 042133MR--POSTDOC of the Universidad de Santiago de Chile.

\bibliographystyle{amsplain}


\end{document}